\documentclass[11pt,dvips]{article}
\usepackage{latexsym}
\usepackage{amsmath,amsthm,amsfonts,amssymb}
\usepackage{enumerate} 
\usepackage{graphicx}
\usepackage{pst-all}  
\usepackage[a4paper,textwidth=15cm,textheight=25cm]{geometry} 
\usepackage[margin=1cm]{caption}
 
\newtheorem{thm}{Theorem}[section]
\newtheorem{dfn}[thm]{Definition} 
\newtheorem{cor}[thm]{Corollary}
\newtheorem{prop}[thm]{Proposition} 
\newtheorem{lem}[thm]{Lemma} 
 
\newtheorem*{no*}{Notation}
\theoremstyle{remark}
\newtheorem*{rem*}{Remark}
\newtheorem{rem}[thm]{Remark}
\newtheorem*{example*}{Example}
\newtheorem{example}[thm]{Example}

\makeatletter
\edef\@tempa#1#2{\def#1{\mathaccent\string"\noexpand\accentclass@#2 }}
\@tempa\rond{017}
\makeatother

\newcommand{\es}{\emptyset}
\renewcommand{\phi}{\varphi} 
\newcommand{\m} {^{-1}} 
\newcommand {\cale} {{\mathcal {E}}}   
\newcommand {\calp} {{\mathcal {P}}}   
\newcommand{\Out} {\mathop{\mathrm{Out}}}
\newcommand{\Aut} {\mathop{\mathrm{Aut}}}
\newcommand {\Z} {{\mathbb {Z}}}
\newcommand {\Q} {{\mathbb {Q}}}
\newcommand {\R} {{\mathbb {R}}}
\newcommand{\inc}{\subset}
\newcommand{\bo}{\partial}
\newcommand{\rk}{\mathop{\mathrm{rk}}}
\newcommand{\ov}{\overline }

\usepackage{pdfsync}

\begin{document}

\title{Generalized Baumslag-Solitar groups: rank and finite index subgroups}
\author{ Gilbert Levitt}
\date{}

\maketitle

\begin{abstract}
A generalized Baumslag-Solitar (GBS) group is a finitely generated group    acting on a tree  with infinite cyclic edge and vertex stabilizers. We show how to  
determine effectively the rank (minimal cardinality of a generating set) of a GBS group; as a consequence,  one can compute the rank of the mapping torus of a finite order outer automorphism of a free group $F_n$. We also show that the rank of a finite index subgroup of a GBS group $G$ cannot be smaller than the rank of $G$. We determine which GBS groups are large (some finite index subgroup maps onto $F_2$), and we solve   the commensurability problem (deciding whether two groups have isomorphic finite index subgroups) in a particular family of GBS groups.
\end{abstract} 

\section{Introduction and statement of results}

This paper
 studies  
 generalized Baumslag-Solitar (GBS) groups. These are finitely generated groups $G$ acting on a tree $T$ with infinite cyclic edge and vertex stabilizers (equivalently, fundamental groups of finite graphs of groups $\Gamma$ with all vertex and edge groups $\Z$). Basic examples are provided by the Baumslag-Solitar groups $BS(m,n)=\langle a,t\mid ta^mt\m=a^n\rangle$, known for sometimes being non-Hopfian \cite{ BS, CL}, and one-relator groups with non-trivial center \cite{Pi}. 

GBS groups have been studied in particular in relation with JSJ decompositions \cite{FoJSJ},  quasi-isometries \cite{Wh}, automorphisms \cite{Cl,Le}, cohomological dimension \cite{Kr}, the Haagerup property \cite{CV}, Bredon cohomology \cite{DP}, one-relator groups \cite{Co,Mc,Pi}. Solving the conjugacy problem in GBS groups is possible but non-trivial \cite{BeC}, while the isomorphism problem for GBS groups is open in general (see \cite{ CFIso,FoGBS,Le} for partial results).

Our first result is a computation of the rank (minimal cardinality of a generating set). The rank is a basic invariant of groups, still there are very few families of groups for which the rank is known (and the rank of a (Gromov)-hyperbolic group is not even computable \cite{Sh}). It is therefore somewhat surprising that there is a complete and explicit way  of computing the rank $\rk(G)$ of a GBS group $G$. 

A GBS group has standard generating sets, with one generator $a_v$ per vertex of the graph of groups $\Gamma$, and one stable letter $t_\varepsilon$ for each edge outside of a given maximal subtree. It turns out that the rank is achieved by a subset of any standard generating set, and that there is a simple combinatorial criterion to decide which subsets of a standard set are themselves generating sets. 

The basic notion involved in this criterion is that of  a \emph{plateau}, which we shall   explain after considering two  simple examples (see also Definition \ref{plato}).  

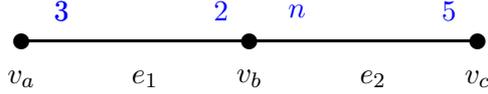
\begin{figure}[h]
\setlength{\unitlength}{1mm}
\begin{picture}(60,20)(-5,9)
   \multiput(40,16)(30,0){3}{\circle*{2} }
  \thicklines
 \put(40,16) 
  {\line(1,0){60}}
    \put(40,11) {\makebox(0,0)
    {$v_a$} }
    \put(70,11){\makebox(0,0){$v_b$}}
    \put(100,11){\makebox(0,0){$v_c$}}
{ \color{blue}
 \put(44,20){\makebox(0,0){$3$}}
  \put(65,20){\makebox(0,0){$2$}}
 \put(75,20){\makebox(0,0){$n$}}
 \put(95,20){\makebox(0,0){$5$}}
 \put(44,20){\makebox(0,0){$3$}}
 }
 {
 \put(55,11){\makebox(0,0){$e_1$}}
  \put(85,11){\makebox(0,0){$e_2$}}
}
  \end{picture}
  \caption{$G_1$ has rank 2 if $n$ is odd, 3 if $n$ is even. 
$\{v_a\}$ is a 3-plateau, $\{v_b\}$ is a 2-plateau for $n$ even,   $\{v_c\}$ is a 5-plateau,  $e_1$ is a  $p$-plateau if $p>3$ divides $n$, $e_2$ is a 2-plateau if $n$ is odd, and $\Gamma$ is a $p$-plateau for  $p>5$   not dividing $n$. If $n$ is odd, every plateau meets  $\{v_a,v_c\}$, so $G_1=\langle a, c \rangle$.
} \label{court}
\end{figure}

Figure \ref{court} represents the labelled graph $\Gamma$ associated to $G_1=\langle a,b,c\mid a^3=b^2, b^n=c^5\rangle$ 
(the graph is labelled  by 3, 2, $n$, 5, which are  indices of edge groups in vertex groups). If $n$ is even, $G_1$ has rank 3 because adding the relation $b^2=1$ maps $G_1$ onto $\Z/3\Z*\Z/2\Z*\Z/5\Z$, which has   rank 3 by Grushko's theorem. On the other hand, if $n$ is odd, $G_1=\langle a, c \rangle$ has rank 2. 

\begin{figure}[h]
\setlength{\unitlength}{1mm}
\begin{picture}(60,20)(10,5)
  \multiput(40,15)(90,0){2}{\circle*{2}}
  \multiput(70,15)(30,0){2}{\red\circle*{2}}
  \thicklines
   \put(40,15) 
  {\line(1,0){29}}
 \put(70,15) 
   {\red\line(1,0){30}}
 \put(101,15) 
  {\line(1,0){30}}
   \put(40,10) {\makebox(0,0)
    {$v_a$} }
    \put(70,10){\makebox(0,0){$v_b$}}
    \put(100,10){\makebox(0,0){$v_c$}}
      \put(130,10){\makebox(0,0){$v_d$}}
{\blue
 \put(44,20){\makebox(0,0){$3$}}
  \put(65,20){\makebox(0,0){\red$2$}}
 \put(75,20){\makebox(0,0){$3$}}
 \put(95,20){\makebox(0,0){$7$}}
 \put(44,20){\makebox(0,0){$3$}}
  \put(105,20){\makebox(0,0){\red$10$}}
 \put(125,20){\makebox(0,0){$5$}}
}
 \put(55,10){\makebox(0,0){$e_1$}}
  \put(85,10){\makebox(0,0){\red$e_2$}}
 \put(115,10){\makebox(0,0){$e_3$}}
\end{picture}
\setlength\belowcaptionskip{-0.3cm}
\caption{$\{v_a\}$ is a 3-plateau, $\{v_d\}$ is a 5-plateau, {\red $e_2$} is a 2-plateau, $e_1\cup e_2$ and $\{v_d\}$ are 5-plateaux, $e_3$ is a 7-plateau, and $\Gamma$ is a $p$-plateau for $p>7$. Every plateau meets $\{v_a,v_b,v_d\}$ and $\{v_a,v_c,v_d\}$, so $G_2=\langle a,b,d\rangle=\langle a,c,d\rangle$.
}\label{long}
\end{figure}
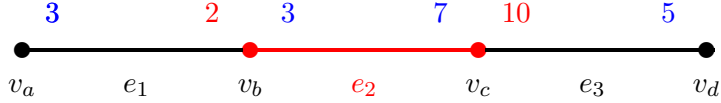

The group $G_2=\langle a,b,c,d\mid a^3=b^2, b^3=c^7, c^{10}=d^5\rangle=\langle a,b,d\rangle=\langle a,c,d\rangle$ of Figure \ref{long} has rank 3: adding the relations $b^2=c^2=1$  maps $G_2$ onto $\Z/3\Z*\langle b,c\mid b^3=c^7, b^2=c^2=1\rangle*\Z/5\Z$, which has rank 3 because the middle group is $\Z/2\Z$. 

This $\Z/2\Z$ is associated to the edge $e_2$, which we  call a \emph{2-plateau}: 2 is a prime, $e_2$ is a connected non-empty subgraph $P$, the labels  carried by edges in $P$    (3 and 7) are not divisible by 2, but to exit the plateau one must pass by a label (2 or 10) which is divisible by 2 (see Figures \ref{court} and \ref{long} for examples, and Definition \ref{plato} for a precise definition).  A plateau $P$ is \emph{proper} if $P\ne\Gamma$.

We will show that one obtains a generating set of minimal cardinality $\rk(G)$ by removing   certain generators $a_v$ from  any standard generating set $\{(a_v),(t_\varepsilon)\}$. The set remains a  generating set provided one keeps  at least one $v$ in each plateau (see Figures \ref{court} and \ref{long}). Thus:

\begin{thm}[Theorem \ref{rang}] \label{main1}
 Let $G$ be a GBS group represented by a labelled graph $\Gamma$. The rank of $G$ is $\beta(\Gamma)+\mu(\Gamma)$, where $\beta(\Gamma)$ is the first Betti number and $\mu(\Gamma)$ is the minimal cardinality of a set of vertices meeting every plateau.
\end{thm}

On Figure \ref{cerp}, $G=\langle a,b,c,t\mid a^3=b^5, b^2=c^3, tc^5t\m=a^2\rangle=\langle a,b ,t \rangle=\langle a, c,t \rangle=\langle b,c,t \rangle$. There are four  plateaux: $\Gamma$ and each of the   edges, so $\mu(\Gamma)=2$ and $\rk(G)=1+2=3$ by the theorem. In this example Grushko's theorem is not sufficient to prove $\rk(G)\ge3$, and we have to use Dunwoody's folding sequences \cite{Du}.

\psset{unit=1pt}
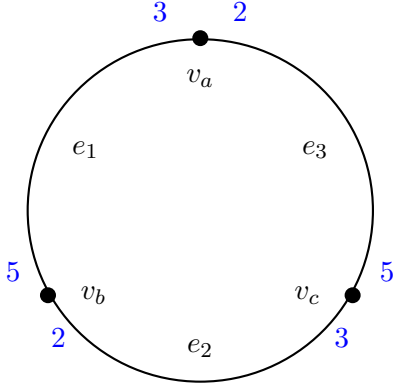
\begin{figure}[h]
\begin{center}
\begin{pspicture}(100,140) 
   \pscircle( 50,50) {65}
    \pscircle*( 50,115) {3}
    \pscircle*( 107,18) {3}
    \pscircle*( -7,18) {3}
{\blue
 \rput(-3,2) {$ 2$}
  \rput(-20,27) {$ 5$}
   \rput(103,2) {$ 3$}
    \rput(120,27) {$ 5$}
  \rput(65,125) {$ 2$}
   \rput(35,125) {$ 3$}
   }   
   \rput(50,100) {$ v_a$} 
      \rput(10,18) {$ v_b$} 
         \rput(90,18) {$ v_c$} 
         \rput(7,73) {$e_1$} 
    \rput(50,-2) {$e_2$} 
     \rput(93,73) {$e_3$}     
\end{pspicture}
\end{center}
\setlength\belowcaptionskip{-0.3cm}
\setlength\abovecaptionskip{0.8cm}
\caption{ $e_1$ is a\ 2-plateau, $e_2$ is a 5-plateau, $e_3$ is a 3-plateau, $\Gamma$ is a $p$ plateau for $p>5$, and $G= \langle a,b ,t \rangle=\langle a, c,t \rangle=\langle b,c,t \rangle$.
} \label{cerp}
\end{figure}

\begin{cor}[Corollary \ref{rmt}]
One may compute the rank of the mapping torus of a finite order outer automorphism of a finitely generated free group $F_n$.
\end{cor}

Indeed, these mapping tori are exactly the GBS groups with a non-trivial center (see Proposition \ref{susp}).
 
 Using Theorem \ref{main1}, we can show:
 
\begin{thm}[Theorem \ref{indfi}]\label{main2}
 If $\ov G$ is a finite index subgroup of a GBS group $G$, then $\rk(\ov G)\ge\rk(G)$.
\end{thm}

This property holds in free groups and surface groups, but not in arbitrary hyperbolic groups. As pointed out by I.\ Kapovich, the Rips construction \cite{Ri} (applied to finite groups) yields hyperbolic groups of arbitrarily large rank containing a 2-generated subgroup of finite index.

A group   $\ov G$ as in the theorem is represented by a labelled graph $\ov \Gamma$, with a map $\pi:\ov\Gamma\to \Gamma$ sending edge to edge and vertex to vertex. It is easy to show $\beta(\ov\Gamma)\ge \beta(\Gamma)$. If some component of $\pi\m(P)$ is a plateau whenever $P\inc\Gamma$ is a plateau,   the result  is clear.   Unfortunately this is not true in general, so that $\mu(\ov\Gamma)$ may be smaller than $\mu(\Gamma)$, and we   have to show that this is compensated for by an increase in $\beta$. This is not too hard to do when $\Gamma$ contains no proper 2-plateau, but 2-plateaux bring technical complications. See for instance Figure  \ref{disp} for a subgroup  of index 2 of
 $G=\langle a,b ,t\mid a^2=b^2,  tb^3t\m=b^5\rangle$.  The graph $\Gamma$ representing $G$
 contains  two proper 2-plateaux (the terminal vertex and the ellipse), but   $\ov\Gamma$ contains no proper plateau.
  
\begin{figure}[h]
\begin{center}
\begin{pspicture}(100,200)
   \rput(0,20){
\psellipse(120,0)(60,30)
     \pscircle*( 60,0) {3}
    \pscircle*( -60,0) {3}
  \psline(-60,0)(60,0)
 \rput(200,0){$ \Gamma$}   
 {\blue
   \rput(75,10) {$ 5$}
   \rput(75,-10) {$ 3$}
    \rput(45,10) {$ 2$}
      \rput(-45,10) {$ 2$}
   }   
   }
      \psline[arrowsize=5pt] {->}(60,120)(60,60)
   \rput(67,90){$\pi$}
    \rput(0,160){
   \psellipse(0,0)(60,30)
\psellipse(120,0)(60,30) 
    \pscircle*( 60,0) {3}
    \pscircle*( -60,0) {3}
 \rput(200,0){$\ov\Gamma$}    
{\blue
  \rput(75,10) {$ 5$}
   \rput(75,-10) {$ 3$}
    \rput(45,10) {$ 1$}
  \rput(45,-10) {$ 1$}
      \rput(-65,10) {$ 1$}
  \rput(-65,-10) {$ 1$}
   }}
   \end{pspicture}
\end{center}
\caption{}
\label{disp}
\end{figure}
 
 If $\Gamma$ represents $G$, and $\pi: \Gamma'\to\Gamma$ is a finite covering, one may lift the labels from $\Gamma$ to $ \Gamma'$, and $ \Gamma'$ represents a finite index subgroup of $G$. When $\Gamma$ contains no proper plateau, every finite index subgroup of $G$ is isomorphic to one obtained by this construction (Proposition \ref{revet}).  On the other hand, when $\Gamma$ contains a proper $p$-plateau $P$, there are  other finite index subgroups, obtained  by taking a   covering of degree $p$ branched over $P$  (on Figure  \ref{disp}, the branching is over the union of two 2-plateaux). 
 As an application, we identify which GBS groups are large (i.e.\ they have a finite index subgroup mapping onto a free group $F_2$):
 
 \begin{thm}[Proposition \ref{larg}]\label{lar}
A non-cyclic GBS group $G$ is large if and only if  it cannot be represented by a labelled graph homeomorphic to a circle and containing no proper plateau. 
\end{thm}
 
 When $\Gamma$ is a circle, with labels denoted $x_i$ and $y_i$ alternatively,  
 the absence of a proper plateau is equivalent to $\prod x_i$ being coprime with $\prod y_i$.
  Theorem \ref{lar} was proved independently by T.\ Mecham \cite{Me}; see \cite{EP} for the case of Baumslag-Solitar groups. We   show in \cite{Les} that \emph{$G$ is large if and only if it is not a quotient of $BS(m,n)$ with $m,n$ coprime}.
  
There may be infinitely  many different labelled graphs $\Gamma$ representing a given GBS group $G$ (this is why the isomorphism problem is difficult). Uniqueness (up to sign changes and restricting to reduced graphs) holds  when $\Gamma$ is strongly slide-free \cite{Fodef}: no two labels near a given vertex divide each other.   Using this, we show:
 
  \begin{thm}[Theorem \ref{co}] \label{com}  The commensurability problem is solvable among 
  GBS groups  represented by labelled graphs which are strongly slide-free and contain no proper plateau. 
  \end{thm}

 The commensurability problem consists in deciding whether two given groups contain isomorphic finite index subgroups. Among   Baumslag-Solitar groups $BS(m,n)$, it is  solved  when $m$ and $n$  are coprime  \cite{Wh},  or equal in absolute value,   but the general case is open. Note that   all    GBS groups which are virtually $F_n\times \Z$ for some $n $ are commensurable (provided $n>1$); these groups are called unimodular (see Section \ref{prel}), in particular every GBS group with non-trivial center is unimodular.
 
Theorem \ref {com} applies in particular to the non-large groups of Theorem \ref{lar}. We show that any GBS group has a finite index subgroup represented by a labelled graph with no proper plateau (Proposition \ref{papla}), but unfortunately this graph is not strongly slide-free in general. 

Deciding commensurability in Theorem \ref{com} amounts to deciding whether the labelled graphs  
have a common finite covering; Leighton's graph covering theorem \cite{Lei, Ne} ensures that this is possible. 
 
I am grateful to M.\ Forester for explaining how to visualize finite index subgroups, to R.\ Weidmann  for suggesting the use of folding sequences,  and to W.\   Dicks  for mentioning Leighton's theorem and the reference \cite{Ne}.
 This work was partly supported by 
ANR-07-BLAN-0141-01 and ANR-2010-BLAN-116-03.  
 
\section{Preliminaries} \label{prel}

We refer to \cite{FoJSJ,FoGBS,Le} for basic facts about GBS groups.

GBS groups are represented by labelled graphs. 
A \emph{labelled graph} is a finite   graph $\Gamma$ where each oriented edge $e$ has a label $\lambda_e$, a nonzero integer (possibly negative). We denote by $V$ the set of vertices of $\Gamma$, and by $\cale$ the set of non-oriented edges. We view a non-oriented edge as $\varepsilon=(e,\tilde e)$, where $\tilde e$ is $e$ with the opposite orientation. We denote by $v=o(e)$ the origin of $e$, and by $E_v$ the set of oriented edges with origin $v$. The cardinality $ | E_v | $ of $E_v$ is the \emph{valence} of $v$. A vertex is \emph{terminal} if it has valence one. We say that $\lambda_e$ is the \emph{label of $e$ near the vertex $o(e)$}, and that $\lambda_e$ is a label \emph{carried by $e$ or $\varepsilon$}. There are $ | E_v | $ labels near a vertex $v$. 

We write $m\wedge n$ for the greatest common divisor (gcd) of two integers, with the convention that $m\wedge n>0$ regardless of the sign of $m,n$. On the other hand, $lcm(m,n):=\frac{mn}{m\wedge n}$ may be negative.
 
A connected labelled graph defines a graph of groups. All edge and vertex groups are $\Z$, and the inclusion from the edge group $G_e$ to the vertex group $G_{o(e)}$ is multiplication by $\lambda_e$. The fundamental group $G$ of the graph of groups  is the \emph{GBS group represented by $\Gamma$} (we do not always assume that $\Gamma$ is connected, but we implicitly do so  whenever we refer to the group it represents). The group $G$  may be presented as follows. 

Choose a maximal subtree $\Gamma_0\inc \Gamma$. There is one generator $a_v$ for each vertex $v\in V$, and one generator $t_\varepsilon$ (stable letter) for each $\varepsilon$ in $\cale_0$, the set of 
non-oriented edges  
not in $\Gamma_0$. Each non-oriented edge $\varepsilon=(e,\tilde e)$ of $\Gamma$ contributes a relation $R_\varepsilon$. The relation is $(a_{o(e)})^{\lambda_e} = (a_{o(\tilde e)})^{\lambda_{\tilde e}}$ if $\varepsilon$ is   in $\Gamma_0$, and $t_\varepsilon (a_{o(e)})^{\lambda_e}t_\varepsilon\m = (a_{o(\tilde e)})^{\lambda_{\tilde e}}$ if $\varepsilon$ is  not  in $\Gamma_0$ (replacing $e$ by $\tilde e$ amounts to replacing $t_\varepsilon$ by its inverse). This will be called a \emph{standard presentation} of $G$, and the generating set  will be called a \emph{standard generating set} (associated to $\Gamma$ and $\Gamma_0$).

A GBS group is \emph{elementary} if it is isomorphic to $\Z$, or $\Z^2$, or the Klein bottle group $K=\langle x,y\mid x^2=y^2\rangle = \langle a,t\mid tat\m=a\m\rangle$. These are the only virtually abelian GBS groups, and they have very special properties. Unless mentioned otherwise, our results apply to all GBS groups, but 
 we do not always provide proofs for elementary groups. 

A   labelled graph $\Gamma$ is \emph{minimal} if its Bass-Serre tree $T$  contains no proper $G$-invariant subtree; this is equivalent to no label near a terminal vertex being equal  to $\pm1$. The graph $\Gamma$ is \emph{reduced} \cite{Fodef} if any edge $e$ such that $\lambda_e=\pm1$ is a loop ($e$ and $\tilde e$ have the same origin). Any labelled graph may be made reduced by a sequence of elementary collapses (see  \cite{Fodef});  
these collapses do not change $G$. 

The group $G$ represented by $\Gamma$ does not change if one changes the sign of all labels near a given vertex $v$, or if one changes the sign of the labels $\lambda_e$, $\lambda_{\tilde e}$ carried by a given non-oriented edge $\varepsilon$. These will be called \emph{admissible sign changes}. 

 In general, there may be infinitely many reduced labelled graphs representing a given $G$.
  One case when uniqueness holds (up to admissible sign changes) is when $\Gamma$ is \emph{strongly slide-free}  (\cite{Fodef}, see also \cite{Gu}): if $e$ and $e'$ are edges with the same origin, $\lambda_e$ does not divide $\lambda_{e'}$.

We denote by $\beta(\Gamma)$ the first Betti number of $\Gamma$. If $G$ is non-elementary, all labelled graphs representing 
$G$ have the same $\beta$, and we sometimes denote it by $\beta(G)$.
 
Let $G$ be represented by $\Gamma$. 
An element $g\in G$ is \emph{elliptic} if it fixes a point in the Bass-Serre tree $T$ associated to $\Gamma$, or equivalently if some conjugate of $g$ belongs to a vertex group of $\Gamma$. All elliptic elements are pairwise  \emph{commensurable} (they have a common power). An element which is not elliptic is \emph{hyperbolic}. If $G$ is non-elementary, ellipticity or hyperbolicity does not depend on the choice of $\Gamma$ representing $G$.

The quotient of $G$ by the subgroup generated by all elliptic elements may be identified with the topological fundamental group $\pi_1^{top}(\Gamma)$ of the graph $\Gamma$, a free group of rank $\beta(\Gamma)$.   
In particular, any generating set of $G$ contains at least $\beta(\Gamma)$ hyperbolic elements. 

If $G$ is non-elementary, there is a \emph{modular homomorphism} $\Delta_G:G\to\Q^*$ associated to $G$. It may be characterized as follows: given any non-trivial elliptic element $a$, there is a non-trivial relation $ga^m g\m=a^n$,
and
$\Delta_G(g)=\frac mn$ (the numbers $m,n$  may depend on the choice of $a$, but $m/n$ does not).

Elliptic elements have modulus 1, so $\Delta_G$ factors through $\pi_1^{top}(\Gamma)$ for any $\Gamma$ representing $G$. One may define $\Delta_G$ directly in terms of loops in $\Gamma$: if $\gamma\in\pi_1^{top}(\Gamma)$ is represented by an edge-loop $(e_1,\dots,e_m)$, then $\Delta_G(\gamma)=\prod\frac{\lambda_{e_i}}{\lambda_{\tilde e_i}}$. Note that $\Delta_G$ is trivial if $\Gamma$ is a tree.  

$\Delta_G$ is trivial if and only if the center of $G$ is non-trivial. In this case the center is   cyclic and only contains  elliptic elements (see \cite{Le}). Moreover, there is an epimorphism $G\to\Z$ whose kernel contains no non-trivial elliptic element (see Proposition 3.3 of \cite{Le}). 

A non-elementary $G$ is \emph{unimodular} if the image of $\Delta_G$ is contained in $\{1,-1\}$. This is equivalent to $G$ having a normal infinite cyclic subgroup, and also to $G$ being virtually $F_n\times\Z$ for some $n\ge2$. 

Given any $\Gamma$, one may perform
  admissible sign changes so that at most $\beta(\Gamma)$ labels are negative (edges in some maximal subtree only carry positive labels, edges not in the subtree carry at most one negative label). If $\Delta_G$ only takes positive values, one may make all labels positive. 

 \section{Computing the rank}\label{rg}
 
 Let $G$ be a GBS group. In this section we fix a (connected) labelled graph $\Gamma$ representing $G$.  We assume that it is finite, but not necessarily minimal or reduced. 

\begin{dfn} [plateau, plateaunic number]  \label{plato}
Let $p$ be a prime number. A non-empty connected subgraph $P\inc\Gamma$ is a \emph{$p$-plateau} if the following condition holds for every oriented edge $e$ with $v=o(e)$ belonging to $P$:
the label $\lambda_e$ of $e$ near $v$ is divisible by $p$ if and only if $e$ is not contained in $P$.  We say that $P$ is a \emph{plateau} if it is a $p$-plateau for some $p$. It is \emph{proper} if $P\ne\Gamma$.

The \emph{plateaunic number} $\mu(\Gamma)$ is the minimal cardinality of a set of vertices meeting every plateau.
\end{dfn}

Let $P$ be a $p$-plateau. Given $e=vw$, there are four possibilities.
If $e\inc P$, none of the labels $\lambda_e,\lambda_{\tilde e}$ is divisible by $p$. 
If $v,w\in P$ but $e$ is not contained in $P$, both labels $\lambda_e,\lambda_{\tilde e}$ are divisible by $p$. 
If  
 $v\in P$ and $w\notin P$ (or vice versa), the label of $e$ near $v$ is divisible by $p$. 
If $v,w\notin P$, there is no restriction. 

 If $v$ is a vertex, $\{v\}$ is a $p$-plateau if and only if $p$ divides every label near $v$. A terminal vertex with label $\ne\pm1$ is a plateau. Two $p$-plateaux are disjoint or equal. The whole graph $\Gamma$ is a $p$-plateau for all but finitely many $p$, so $\mu(\Gamma)\ge1$.
  
 \begin{thm} \label{rang}
 Let $G$ be a GBS group represented by a labelled graph $\Gamma$. 
 \begin{itemize}
 \item
 Every standard generating set contains a generating   set of cardinality $\rk(G)$.
 \item
 The rank of $G$ equals $\beta(\Gamma)+\mu(\Gamma)$, where $\beta(\Gamma)$ is the first Betti number of $\Gamma$ and $\mu(\Gamma)$ is its plateaunic number (see Definition \ref{plato}).
\end{itemize}
 \end{thm}
 
 In particular, the rank of $G$ is computable from any labelled graph representing $G$. It follows from Proposition \ref{menage} below that one obtains  a generating set of 
 cardinality $\rk(G)$ by finding a  set of vertices $V_1\inc V$ of cardinality $\mu(\Gamma)$ meeting every plateau, and   deleting   generators $a_v$ for $v\notin V_1$ from any standard generating set $\{(a_v)_{v\in V}, (t_\varepsilon)_{ \varepsilon\in \cale_0}\}$. 
  
 \begin{rem}
  $\beta(\Gamma)+\mu(\Gamma)$ is an invariant: it only depends on $G$. For $G$ non-elementary  it is well-known that $\beta(\Gamma)$ is an invariant, and it is not hard to deduce directly from \cite{Fodef} or \cite{CFWh}  that $\mu(\Gamma)$ is also an invariant. 
\end{rem}

  The remainder of this section is devoted to the proof of  Theorem \ref{rang}, so we fix $G$ and $\Gamma$. 
  
Given a prime number $p$ and a $p$-plateau $P\inc\Gamma$, fix a maximal subtree $P_0\inc P$, and a maximal subtree $\Gamma_0\inc \Gamma$ with $\Gamma_0\cap P=P_0$. Consider the associated standard presentation, with generators $a_v$ and $t_\varepsilon$. We define a quotient $G'$ of $G$ by adding the relations $a_v=1$ for $v\notin P$, and $(a_v)^p=1$ for $v\in P$.

\begin{lem} \label{grq}
The rank of $G'$ is $\beta(\Gamma)+1$.
\end{lem}

\begin{proof}
 Consider an edge $\varepsilon=(e,\tilde e)$ not contained in $P$. The elements $(a_{o(e)})^{\lambda_e}$ and $(a_{o(\tilde e)})^{\lambda_{\tilde e}}$ are killed in $G'$ since $\lambda_e$ is divisible by $p$ if $o(e)\in P$, so  we may remove the relation $R_\varepsilon$ from the presentation of $G'$. 

For simplicity, we first consider the case when  $\Gamma$ is a tree. Then $G'$ is generated by the elements $a_v$ for $v\in P$, and $(a_v)^p=1$ in $G'$. The other relations come from edges of $P$, they are of the form $(a_{o(e)})^{\lambda_e}=(a_{o(\tilde e)})^{\lambda_{\tilde e}}$. The key point here  (as in the group $G_2$ of Figure \ref{long}) is that no exponent  $\lambda_e $ (or ${\lambda_{\tilde e}}$) is divisible by $p$, so multiplication by $\lambda_e$ defines an automorphism of $\Z/p\Z$. Thus $G'$ is the fundamental group of a graph of groups with underlying graph $P$, vertex and edge groups $\Z/p\Z$, and inclusions given by multiplication by the $\lambda_e$'s. It follows that $G'$ is isomorphic to $\Z/p\Z$, so has rank $1=\beta(\Gamma) +1$.

The general case is similar. Since $\Gamma_0\cap P=P_0$, there are $\beta(\Gamma)-\beta(P)$ edges not contained in either $P$ or $\Gamma_0$. The associated stable letters are not involved in any relation, so $G'$ is the free product of a free group of rank $\beta(\Gamma)-\beta(P)$ with   $G''=\langle (a_v)_{v\in P}, (t_\varepsilon)_{\varepsilon\inc \ov{P\setminus P_0}}\rangle$. As in the previous case, $G''$ is the fundamental group of a graph of $\Z/p\Z$'s based on $P$. It maps onto the topological fundamental group of $P$ with kernel $\Z/p\Z$, so has rank $\beta(P)+1$. By Grushko's theorem, $G'$ has rank $\beta(\Gamma)-\beta(P)+\beta(P)+1=\beta(\Gamma)+1$.
\end{proof}

\begin{dfn} [Minimally hyperbolic] \label{mh}
A generating set $S$ of $G$ is \emph{minimally hyperbolic} if it contains exactly $\beta(\Gamma)$ hyperbolic elements. 
\end{dfn}

Recall that any generating set contains \emph{at least} $\beta(\Gamma)$ hyperbolic elements. 
Standard generating sets are minimally hyperbolic.

\begin{cor} \label{rencplat}
 Let $S$ be a minimally hyperbolic generating set. For  each elliptic element $s\in S$, fix a vertex $v_s$ of $\Gamma$ such that $s$ has a conjugate contained in the vertex group $G_{v_s}$.
Then every plateau   contains some $v_s$. In particular, $S$ has at least $\beta(\Gamma)+\mu(\Gamma)$ elements. 
\end{cor}

\begin{proof} 
Let $P$ be a $p$-plateau, and let $G'$ be as above. If $P$ contains no $v_s$, every $a_{v_s}$ is killed in $G'$, so every 
 elliptic element  of $S$ is killed in $G'$. Since $S$ is minimally hyperbolic,  $G'$ may be generated by $\beta(\Gamma)$ elements. This contradicts Lemma \ref{grq}.
\end{proof}

The following fact is used in \cite{Les}.  

\begin{cor} \label{divp}
 Let $S$ be a minimally hyperbolic generating set. 
Given a $p$-plateau $P$,   there exist   $s\in S$ and $v \in P$ such that  $s$ is conjugate to a power   $(a_{v })^{q}$ with $q$ not divisible by $p$. 
\end{cor}

\begin{proof}
As in the previous proof, some elliptic $s\in S$ must survive in $G'$. Being elliptic, $s$ is conjugate to some power $(a_ v  )^{q}$ with $v\in V$.  Survival  implies that $v$ is in $P$ and $q$ is not divisible by $p$.
\end{proof}

\begin{rem} \label{op}
  If $S$ is not minimally hyperbolic, let $h(S)$ be the number of hyperbolic elements in $S$. If $P_1,\dots,P_i,\dots P_k$  are  disjoint $p_i$-plateaux   with $k> h(S ) -\beta(\Gamma)$,  there exist $s\in S$, and $i$, such that $s$ is conjugate to $(a_v)^q$ with $v\in P_i$ and $q$ not divisible by $p_i$. To prove this, one defines $G'$ by adding relations $a_v=1$ for $v\notin \cup_i P_i$ and $(a_v)^{p_i}=1$ for $v\in P_i$. It has rank $\beta(\Gamma)+k$, so some elliptic $s\in S$ survives in $G'$. 
\end{rem}

\begin{prop} \label{menage}
 Let $S=\{(a_v)_{v\in V}, (t_\varepsilon)_{\varepsilon\in\cale_0}\}$ be a standard generating set. For $V_1\inc V$, define $S_1\inc S$ by  
 keeping only the $t_\varepsilon$'s and the 
 $a_v$'s  for $v\in V_1$. 
  Then $S_1$ generates $G$ if and only if $V_1$ meets every plateau.  
\end{prop}

\begin{proof} If $S_1$ generates, $V_1$ meets every plateau by Corollary  \ref{rencplat}.
We now suppose that  $G_1=\langle S_1\rangle$ is not $G$, and we construct a plateau $P$ disjoint from $V_1$. We may assume $V_1\ne \es$ and $V_1\ne\Gamma$.

For $v\in V $, define the integer $n_v\ge1$ by
$G_1\cap\langle a_v\rangle=\langle a_v^{n_v}\rangle$; it exists because $V_1\ne\es$ and   the $a_v$'s are all commensurable. Note that   $n_v=1$ if
$v\in V_1$.  Since $G_1\ne G$, we have $n_v>1$ for some $v\notin V_1$, so we may fix a prime $p$ dividing some $ n_v$. For each $v\notin V_1$, let $p^{\delta _v}$ be the maximal power   dividing $n_v$, and let $\delta \ge 1$ be the maximal value of $\delta _v$ (for $v\notin V_1$). 

We now define a subgraph $ \Gamma_1\inc \Gamma$ as follows. Its vertices are the vertices $v\notin V_1$ such that 
$\delta _v=\delta $. An edge $\varepsilon=(e,\tilde e)$ is in $\Gamma_1$ if and only if its vertices are in $\Gamma_1$ and none of $\lambda_e,\lambda_{\tilde e}$ is divisible by $p$. We complete the proof by showing that every component   of $\Gamma_1$ is a $p$-plateau.

This is equivalent to  the following fact:  if $e $ is an edge with $\delta _{o(e)}=\delta $ such that the label $\lambda_e$ of $e$ near $o(e)$ is not divisible by $p$, then $\delta _{o({\tilde e})}=\delta $ and the label $\lambda_{ {\tilde e} }$ near $ o(\tilde e)$ is not divisible by $p$. Let us write $v=o(e)$ and $w=o(\tilde e)$.

Let $a_e=(a_{v})^{\lambda_e}$. Since $p$ does not divide $\lambda_e$, and $\delta_v=\delta$, the smallest $n$ such that $(a_e)^n\in G_1$ is divisible by $p^\delta$.
The standard presentation of $G$ contains a relation $R_\varepsilon$ expressing $a_e$ as   $ (a_{w})^{\lambda_{\tilde e}}$ (possibly conjugated by a stable letter $t_\varepsilon$). 
Let $\displaystyle\theta= \frac {lcm(n_{w},\lambda _{ \tilde e })} {\lambda _{ \tilde e }}=\frac{n_w}{ n_w\wedge\lambda_{\tilde e} }
$.
Since $G_1$ contains $ (a_{w})^{n_{w}}$  and all stable letters,
it contains $(a_w)^{\lambda_{\tilde e}\theta}$ and $(a_e)^\theta$.
It follows that $p^\delta$ divides $\theta$, hence $n_w$, so $\delta_w=\delta$. Since $n_w$ is not divisible by $p^{\delta+1}$, we deduce that   $p$ does not divide $\lambda _{ \tilde e }$.
\end{proof}

Corollary \ref{rencplat} and Proposition \ref{menage} imply:

\begin{cor} \label{2prop}
\begin{itemize}
\item
Every standard generating set contains a generating set of cardinality  $\beta(\Gamma)+\mu(\Gamma)$.   In particular, $\beta(\Gamma)+\mu(\Gamma)\ge\rk(G)$. 
\item The minimal cardinality of a minimally hyperbolic generating set is $\beta(\Gamma)+\mu(\Gamma)$.  
\end{itemize}
\end{cor}

This is clear, recalling that standard generating sets are minimally hyperbolic. 
The following proposition now completes the proof of Theorem \ref{rang}.

\begin{prop} \label{minhyp}
There exists a minimally hyperbolic generating set of cardinality $\rk(G)$. 
\end{prop}

\begin{proof} The proof is by induction on $r=\rk(G)$, with the result clear for $r=1$. We fix a generating set $X$ of cardinality $r  $, which we view as an epimorphism $\pi$ from the free group $F(X)$ to $G$.  Let $T_0$ be the Cayley tree of  $F(X)$, with the natural action of $ F(X)$. 
Let $T$ be the Bass-Serre tree of $\Gamma$, with the natural action of $G$. Trimming $\Gamma$ if necessary, we may assume that $G$ acts minimally on $T$.
 
By Theorem 2.1 of \cite{Du}, there exists a sequence $T_0\overset{f_0}\to T_1\overset{f_1} \to\dots \to T_i \to \dots\overset{f_{n-1}}\to  T_n=T$ of simplicial trees $T_i$ with an action of a group $G_i$ such that:
\begin{enumerate}
\item
$G_0=F(X)$, $G_n=G$, and there are epimorphisms $\pi_i: G_i\to G_{i+1}$, with $\pi_{n-1}\circ\dots\circ\pi_0=\pi$; 
\item
 the restriction of $\pi_{n-1}\circ\dots\circ\pi_i$  to each vertex stabilizer of $T_i$ is injective; 
\item
the map $f_i:T_i\to T_{i+1}$ is equivariant with respect to $\pi_i$; it is either a subdivision or a basic fold followed by  a vertex morphism (see \cite{Du} for definitions).
\end{enumerate}

The action of $G_i$ on $T_i$ is not necessarily minimal. By condition 2 (and equivariance), all vertex and edge stabilizers are cyclic (possibly trivial). In particular, $G_i $ is a free product of GBS groups. 
Note that all groups $G_i$ have rank $r$. If a group acts on  a tree $T$, we say that a generating set is minimally $T$-hyperbolic if no other generating set contains fewer hyperbolic elements. 

Consider the smallest $k$ such that $G_{k+1}$ is a GBS group. 

$\bullet$ We first show that $G_{k+1}$ has a minimally $T_{k+1}$-hyperbolic generating set of cardinality $r$. 
 
 The previous group $G_k$ is a nontrivial free product of (possibly cyclic) GBS groups of rank $<r$. Indeed, define a graph of groups $\hat\Gamma_k$ by  collapsing to a point each edge  of $\Gamma_k= T_k/G_k$ with stabilizer $\Z$. We can write $G_k$ as the free product of GBS vertex groups $H_v$ of $\hat\Gamma_k$, together with  $\beta(\hat\Gamma_k)$ infinite cyclic groups  generated by stable letters. The free product is nontrivial because $G_k$ is not  a GBS group.
 
 By induction on $r$, each $H_v$ has a minimally $T_k$-hyperbolic generating set $S_v$ of cardinality $\rk(H_v)$. 
Consider a generating set $S_k$ of $G_k$  
obtained by adjoining $\beta(\hat\Gamma_k)$ stable letters to the union   of the $S_v$'s.
 By Grushko's theorem, it has $r$ elements. Since each $S_v$ is minimally $T_k$-hyperbolic, exactly $\beta(\Gamma_k)$ elements of $S_k$  are hyperbolic in $T_k$. 

If passing from $T_k$ to $T_{k+1}$ involves no basic fold of type III (see Figure 1 in [Du]), one has $\beta(\Gamma_{k+1})=\beta(\Gamma_k)$. Since $\pi_k(g)$ is elliptic in $T_{k+1}$ if $g$ is elliptic in $T_k$, the image of $S_k$ in $G_{k+1}$ is the desired  minimally $T_{k+1}$-hyperbolic generating set. If there is a fold of type III, so that $\beta(\Gamma_{k+1})=\beta(\Gamma_k)-1$,  
the stable letter associated to the folded loop  (the element $g$ in Figure 1 of  \cite{Du})  becomes elliptic in $T_{k+1}$. Once again $\pi_k(S_k)$ is a minimally $T_{k+1}$-hyperbolic generating set of $G_{k+1}$.

$\bullet$ We now show  by induction on $i\ge k+1$ that $G_i$ has a minimally $T_i$-hyperbolic generating set of cardinality $r$.   This proves the proposition since $G_n=G$. 

%The result is true for $i=k+1$. Assume it is true for some $i$. By Corollary \ref{2prop}, we may get a minimally $T_i$-hyperbolic generating set  $S_i$ of cardinality $r$ for $G_i$ by removing elliptic elements from any standard generating set. As above, the image of $S_i$ in $G_{i+1}$ is the desired generating set:
%  if there is a basic fold of type III, we have $\beta(\Gamma_{i+1})=\beta(\Gamma_i)-1$, but some stable letter in $S_i$ becomes elliptic in $T_{i+1}$. 
  
  The result is true for $i=k+1$. Assume it is true for some $i$. 
  %Let $S_By Corollary \ref{2prop}, we may get 
  Let $S_i$ be a minimally $T_i$-hyperbolic generating set of $G_i$  of cardinality $r$,  obtained by removing elliptic elements from any standard generating set (such an $S_i$  exists by the induction hypothesis and  Corollary \ref{2prop}). As above, the image of $S_i$ in $G_{i+1}$ is the desired generating set:
  if there is a basic fold of type III, we have $\beta(\Gamma_{i+1})=\beta(\Gamma_i)-1$, but the definition of $S_i$ guarantees that some stable letter in $S_i$ becomes elliptic in $T_{i+1}$. 

\end{proof}

\section{Finite order automorphisms of free groups}

We denote by $F_n$ the free group of rank $n$. Given $\Phi\in\Out(F_n)$, we denote by $M_\Phi$ the mapping torus $M_\Phi=F_n\rtimes_\Phi\Z=\langle F_n,t\mid tgt\m=\alpha(g)\rangle$, where $\alpha\in\Aut(F_n)$ is any representative of $\Phi$.

 \begin{prop} \label{susp}
Given a group $G$, the following are equivalent:
\begin{itemize}
\item $G$ is a GBS group with non-trivial center;
\item there exist $n$ and $\Phi\in\Out(F_n)$ of finite order such that $G=M_\Phi$.
\end{itemize}
\end{prop}

\begin{proof} 
We may assume that $G$ is not $\Z$, $\Z^2$, or the Klein bottle group. 

 Let $G$ be a GBS group with non-trivial center $Z$.  Then (see Section \ref{prel})  $G$ is virtually $F_r\times \Z$ for some $r\ge 2$, and  
 there is an epimorphism $\tau:G\to \Z$ which is non-trivial on $Z$. The restriction of $\tau$ to $F_r\times \Z$ is non-trivial on the factor $\Z$, so its kernel is free of finite rank (projecting onto $F_r$ identifies the kernel with a finite index subgroup of $F_r$). 
 It follows that $\ker\tau$ is finitely generated and virtually free, hence equal to some $F_n$  since $G$ is torsion-free. This shows that $G=M_\Phi$ for some $\Phi\in\Out(F_n)$. View the center $Z$ as a subgroup of $M_\Phi$ and write a generator   as $t^kg$ with $k\in\Z$ and $g\in F_n$. One has $k\ne0$, and  $\Phi^k$ is the identity because $Z$ is central.

Conversely, suppose that $\Phi\in\Out(F_n)$ has finite order. We first show that $M_\Phi$ is a GBS group. 
The following simple argument is due to V.\ Guirardel. By \cite{Cu, Kh, Zi},  $\Phi$ is induced by an automorphism of a finite graph $\Lambda$ with fundamental group isomorphic to $F_n$. The group $M_\Phi$ is the fundamental group of a 2-complex  
made of circles (associated to vertices of $\Lambda$) and annuli (associated to edges), so is a GBS group. We now show that the center of $M_\Phi$ is non-trivial.

Let $\alpha\in\Aut(F_n)$ be a representative of $\Phi$. 
There exists $k\ge1$ such that $\alpha^k$ is conjugation by some $g\in F_n$.  If $g$ is trivial, $t^k$ is central.  If not, we consider
the subgroup of $M_\Phi$ generated by $F_n$ and $t^k$. It has finite index and non-trivial center (generated by $t^kg\m$). It follows that the modular map $\Delta_{M_\Phi}$ (see Section \ref{prel})  has finite image, so $M_\Phi$ has  a normal infinite cyclic subgroup $Z$ (containing $t^kg\m$). This subgroup is not contained in $F_n$, so is in fact central: otherwise any non-trivial $z\in Z$ would be conjugate to $z\m$, a contradiction since $z$ maps non-trivially under the natural projection from $M_\Phi$ to $\Z=\langle t\rangle$.
\end{proof}

\begin{example}\label{tet}
 $\Out(F_2)$ contains an element $\Phi$ of order 6. It is represented by an automorphism $f$ of the theta graph (it has 2 vertices joined by 3 edges, and $f$ is a symmetry of order 6). One considers midpoints of edges as  vertices, so that there is one orbit of edges (with period 6) and two orbits of vertices (one with period 3, one with period 2). It follows that  $M_\Phi=\langle a,b\mid a^3=b^2\rangle$.
\end{example}

 \begin{rem*}
It may be shown that $M_\Phi$ is not a GBS group if $\Phi\in\Out(F_n)$ has infinite order:  $M_\Phi$ is residually finite (as a semi-direct product of finitely generated residually finite groups), 
so cannot be a GBS group unless it is virtually $F_r\times \Z$ (see \cite{Les} for a proof);
but then $\Phi$ has finite order (argue as in the proof of the proposition).
\end{rem*}

\begin{cor}
Let $\alpha\in\Aut(F_n)$ and $k\ge2$. If $\alpha^k$ is conjugation by some  $g\in F_n$, then $\alpha(g)=g$.
\end{cor}

\begin{proof}
We have seen in the   proof of Proposition \ref{susp}  that
 $t^kg\m$ is central in $M_\Phi$. In particular it commutes with $t$, so $t$ commutes with $g$. This means $\alpha(g)=g$.
\end{proof}

\begin{rem*}
This result is not specific to $G=F_n$. It holds whenever the outer automorphism defined by $\alpha$  may be represented by a homeomorphism  of finite order of a space $X$ with $\pi_1(X)=G$.
\end{rem*}

\begin{cor} \label{rmt}
One may compute the rank of the mapping torus of a finite order automorphism of a free group $F_n$:   there is an algorithm which, given $\Phi\in\Out(F_n)$ of finite order, computes the rank of $M_\Phi=F_n\rtimes_\Phi\Z$.
\end{cor}

\begin{proof}  
Given $\Phi$ of finite order, we know that $M_\Phi$ is a GBS group. We may find 
 a standard GBS presentation of $M_\Phi$    by applying Tietze transformations to the presentation as a semidirect product, or by arguing as in Example \ref{tet}. We then apply Theorem \ref{rang}.
\end{proof}

\begin{rem*}
 Given $\Phi\in\Out(F_n)$ and $\Psi\in\Out(F_q)$, both of finite order, one may decide whether $M_\Phi$ and $M_\Psi$ are isomorphic \cite{FoGBS}.
\end{rem*}
 
\section{The rank of   finite index subgroups}

This section is devoted to the proof of the following result:

\begin{thm} \label{indfi}
If $G$ is a GBS group, and $\ov  G \inc G$ has finite index, then $\rk(\ov G)\ge\rk(G)$.
\end{thm}

\subsection{A reduction}\label{reduc}

The first step in the proof is to reduce the theorem to a result about graphs (Proposition \ref{keyg}).

We say that  a map between graphs is a \emph{morphism} if it 
sends  vertex to   vertex and   edge to   edge.

We represent $G$ by a labelled graph $\Gamma$. In this section we will assume that $\Gamma$ is reduced (if $\lambda_e=\pm1$, then $e$ is a loop). The group $\ov G$ acts on the Bass-Serre tree $T$ of $\Gamma$  and this yields a graph of groups $\ov \Gamma=T/\ov  G$, with a morphism $\pi:\ov \Gamma\to \Gamma$. We describe $\ov \Gamma$ and $\pi$ using topology. More conceptually,   $\pi$ (and the admissible maps of Section \ref{pfi}) are coverings of graphs of groups in the sense of Bass \cite{Ba}; our point of view is closer to that of Scott-Wall \cite{SW}.
 
 It is standard to  associate  a foliated 2-complex $\Theta$ to $\Gamma$.  One associates a circle $C_v$ to each vertex $v$, an annulus $A_\varepsilon=[0,1]\times S^1$ foliated by circles $\{*\}	\times S^1$ to each non-oriented edge $\varepsilon$, and boundaries of annuli are glued to   circles by maps whose degree (positive or negative)  is given by the labels of $\Gamma$. One recovers $\Gamma$ from $\Theta$ by collapsing each circle to a point.  
 The fundamental group of $\Theta$ is $G$, and $\ov G$ defines a finite covering $\rho:\ov \Theta\to \Theta$ whose degree is the index of $\ov G$. The complex $\ov \Theta$ is also made of circles and annuli,   $\ov \Gamma$ is the corresponding graph, and $\pi$ is induced by $\rho$. 

\begin{dfn}[multiplicity $m_x$, $m_{\ov e}$]
A point $x$ of $\ov \Gamma$ corresponds to a circle of $\ov\Theta$, and the restriction of $\rho$ to this circle is a covering map whose degree we call the \emph{multiplicity} $m_x$ of $x$:
we thus associate    a positive integer $m_x$  to each point $x$ of $\ov \Gamma $. 
All points belonging to the interior of a given edge have the same multiplicity, so we also define the multiplicity $m_{\ov e}$ of an edge $\ov e$ of $\ov \Gamma $. Given a point $u$ in $\Gamma$, the sum $\sum_{x\in\pi\m(u)}m_x$ is constant, equal to the index of $\ov G$; we call it the \emph{total multiplicity} of $\pi$.
\end{dfn}

Algebraically, multiplicities are indices $[G_p:G_p\cap\ov G]$, with $G_p$ the stabilizer of     a point   $p\in T$. 

\begin{lem} \label{loc}
The morphism $\pi:\ov\Gamma\to\Gamma$ satisfies the following condition $(*)$:

Given  an edge   $e$  of $\Gamma$, with origin $v$ and label $\lambda_e$ near $v$, and   $x\in\pi\m(v)$, define  $k_{x,e}$ as the gcd   $m_x\wedge\lambda_e$. Then (see Figure \ref{eto}) there are $k_{x,e}$ edges of $\ov\Gamma$ with origin $x$ mapping to $e$; they each have multiplicity $m_x/k_{x,e}$, and their label near $x$ is $\lambda_e/k_{x,e}$. 
\end{lem}

\begin{figure}[h]
\begin{center}
\begin{pspicture}(100,160) 
 \psline  (0,0)(65,0)
     \pscircle*( 0,0) {3}
 \put(0,-12){$v$}
\put(40,-12){$e$}
\rput(12,12){\blue$\lambda_e$}
      \rput (0,110){ 
\rput(2,25){\blue$\lambda_e/k_{x,e}$}        
   \rput(-10,0) {$x$}   
   \rput(146,10){$k_{x,e}=m_x\wedge\lambda_e$}
   \pscircle*( 0,0) {3}  
       \psline  (0,0)(60,50) 
          \psline  (0,0)(60,30) 
            \psline  (0,0)(60,-30) 
            \psline[linestyle=  dashed, dash=3pt 2pt
            ](55,13)(55,-7)
  \rput*  {90}(100,12){$\underbrace{\phantom{aaaaaaaaaaaaaa}}$    }  
    }
 \rput(-15,80)   
 {
  \pscurve[arrowsize=5pt]  {->}(0,10)  (-10,-30)(0,-60)
    }
   \put(-45,50){$m_x$}   
  \rput(70,68)   
 {
  \pscurve[arrowsize=5pt]  {->}(0,0)  (10,-30)(0,-55)
  \rput(35,-30){$m_x/k_{x,e}$}
    }  
   \end{pspicture}
\end{center}
\setlength\belowcaptionskip{-0.2cm}
\caption{condition $(*)$
} \label{eto}
\end{figure}
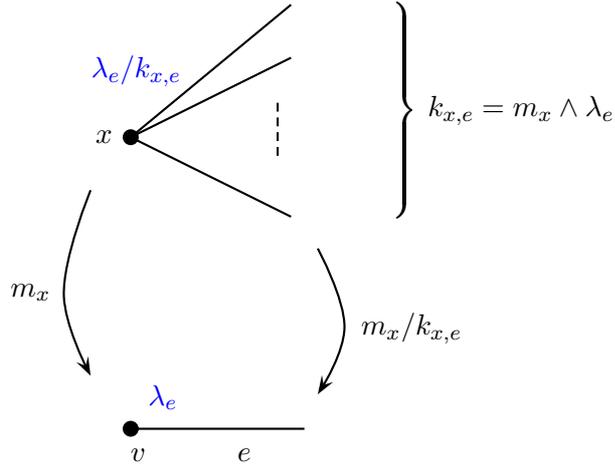

\begin{proof}  
A neighborhood of $C_v$ in the annulus $A_\varepsilon$ corresponding to $e$ has fundamental group $\Z$, and one simply studies its preimage by $\rho$. 
The group $G_x$ carried by $x$ in $\ov\Gamma$ has   index $m_x$ in $G_v$, and $G_e$ has index $\lambda_e$ in $G_v$. Their intersection has index $lcm(m_x,\lambda_e)=m_x\lambda_e/k_{x,e}$ in $G_v$, index $m_x/k_{x,e}$ in $G_e$, and index $\lambda_e/k_{x,e}$ in $G_x$.
\end{proof}

\begin{rem*}
One may also prove this lemma algebraically, by considering intersections of $\ov G$ with vertex and edge stabilizers of $T$.  
\end{rem*}
 
\begin{prop} \label{keyg}
Let $\Gamma$ and $\ov \Gamma$ be   connected labelled graphs. Assume that $\Gamma$ is reduced (if $\lambda_e=\pm1$, then $\Gamma$ is a loop),  and a positive multiplicity is assigned to each vertex and edge of $\ov \Gamma$.
  If there exists a morphism $\pi:\ov \Gamma\to \Gamma$ satisfying $(*)$,  
  then $\beta(\ov \Gamma)+\mu(\ov \Gamma)\ge \beta(\Gamma)+\mu(\Gamma)$. 
\end{prop}

Recall that $\beta $ is the first Betti number, and $\mu $ is the plateaunic number (Definition \ref{plato}).
It follows from Theorem \ref{rang} that this proposition implies Theorem \ref{indfi}. The remainder of this section is devoted to its proof. 

\subsection{The main idea}\label{main}

Let $\pi$ be as in the proposition. It is open: the image of a neighborhood of a vertex $x$ is a neighborhood of $\pi(x)$.
 If $\Gamma'\inc\Gamma$ is a connected subgraph, every component of $\pi\m(\Gamma')$ maps onto $\Gamma'$, and the restriction satisfies $(*)$.
 Note that no label near a terminal vertex of $\ov \Gamma$ is $\pm1$ ($\ov\Gamma$ is minimal, but not necessarily reduced).  
 
 In the situation of Lemma \ref{loc}, let $\ov e$ be an edge of $\ov \Gamma $   with origin $x$ mapping onto $e$. We call $\ov  e$ a \emph{lift} of $e$ (at $x$).  
 
Let $p$ be a prime. 

 If $\lambda_e$ is not divisible by   $p$ (in particular if $e $ is contained in a $p$-plateau), the label $\lambda_{\ov  e}$ of $\ov e$ near $x$ is not divisible by $p$. Moreover $k_{x,e}$ is not divisible by $p$, so  $m_x$ and  $m_{\ov  e}$ are both divisible by $p$ or both non divisible by $p$.  
 
 If $\lambda_e$ is   divisible by $p$,   note that $k_{x,e}$ is divisible by $p$ if and only if $m_x$ is. If $m_x$ is not divisible by $p$, then $\lambda_{\ov  e}$ is divisible by $p$.

It is easy to check that $\beta(\ov \Gamma)\ge \beta(\Gamma)$ (see Lemma \ref{simp2}). Suppose that, for  each plateau $P\inc \Gamma$, some component of $\pi\m(P)$ is a plateau. Then $\mu(\ov \Gamma)\ge \mu(\Gamma)$ and the result is clear. We therefore consider a $p$-plateau $P$, and we study 
$\pi\m(P)$. We may assume $P\ne\Gamma$.

If there exists a vertex $x\in\pi\m(P)$ with $m_x$ not divisible by $p$, it follows from previous observations that the same holds for all points in the component $\ov  P$ of  $\pi\m(P)$ containing $x$, and $\ov  P $ is a $p$-plateau. 

We therefore assume that
$m_x$ is divisible by $p$ for every vertex $x\in\pi\m(P)$.  Since $P\ne\Gamma$, there is an
edge $e$ with origin $v\in P$ which  is not contained in $P$.  
 If $x\inÊ\pi\m(v)$, the number of lifts of $e$ at $x$ is divisible by $p$ since both $m_x$ and $\lambda_e$ are. In particular, denoting by $d_v$ the valence of a vertex, we have $d_x\ge d_v+p-1$.

We conclude that, \emph{if $P$ is a $p$-plateau such that no component of $\pi\m( P)$ is a $p$-plateau,   there exist  $v\in P$ and $x\in \pi\m(v)$ such that the valence of $x$ satisfies  $d_x\ge d_v+p-1$}. 

The main idea of the proof now is the following: the existence of such $P$'s may cause $\mu(\ov \Gamma)$ to be smaller than $\mu(\Gamma)$, but the increase in valence will force $\beta(\ov \Gamma)$ to be larger than $\beta(\Gamma)$, so that $\beta(\ov \Gamma)+\mu(\ov \Gamma)\ge \beta(\Gamma)+\mu(\Gamma)$ does hold (recall the formula $   \beta(\Gamma) =1+\frac12\sum_v (d_v-2)$ for the first Betti number of a connected graph $\Gamma$). 

If $p>2$ for all $P$'s as above, we have  $d_x\ge d_v+2$ and the proof is not too hard, as we shall now explain. On the other hand, if $p=2$, we only have $d_x\ge d_v+1$ and this makes the proof much more complicated. 

\subsection{The simple case}\label{sim}

\begin{no*}
{} \ \newline  
$\beta$: first Betti number (we write $\ov \beta$, $\beta'$ for the first Betti number of $\ov \Gamma$, $\Gamma'$, etc.)

\noindent $\mu$: plateaunic number 

\noindent $t$: number of terminal vertices  

\noindent $d_v$: valence of a vertex

\noindent $V$: vertex set  

\noindent $E_v$: oriented edges with origin $v$

\noindent $\ov\Gamma_v$: preimage $\pi\m(v)$

\noindent $m_x,m_{\ov e}$: multiplicity of a vertex, an edge

\end{no*}

The following   lemma is left to the reader.

\begin{lem} \label{simp} Let $\Gamma$ be a finite connected graph, with vertex set $V$. Then:
  \begin{eqnarray*} 
 \beta&=&1+\sum_{v\in V}(\frac{d_v}2-1) \\	
 \beta+t&=&1+\sum_{v\in V} | \frac{d_v}2-1 |   \\
 \beta+\frac t2&=&1+\sum_{v\in V}\max(\frac{d_v}2-1,0). 
  \end{eqnarray*} 
    \qed 
\end{lem}

\begin{lem}\label{simp2} Let $\pi:\ov \Gamma\to \Gamma$ be an open morphism between finite connected graphs.   
Then:
 \begin{eqnarray*} \beta&\le& \ov \beta \\
 \beta+\frac t2&\le& \ov \beta+\frac {\ov t}2  
   \end{eqnarray*} 
   ($\beta,t$ refer to $\Gamma$, and $\ov \beta$, $\ov t$ refer to $\ov \Gamma$).
\end{lem}

\begin{proof}
 Consider any loop $\gamma$ in $\Gamma$. Since $\pi$ is open, there exists a loop   
 in $\ov \Gamma$ projecting onto a power of $\gamma$. This proves the first inequality.
 
For the other inequality, write  $$\max(\frac{d_v}2-1,0)  \le \sum_{x\in \ov\Gamma_v} \max(\frac{d_x}2-1,0)$$ for $v\in V$ and use the third equality of the previous lemma.
\end{proof}

Let $\pi$ be as in Proposition \ref{keyg}.
\begin{dfn}[$\Delta_v$]
For $v\in V$, define 
$$\Delta_v=\sum_{x\in\ov\Gamma_v}  | \frac{d_x}2-1  |- | \frac{d_v}2-1  |.$$
\end{dfn}

This is non-negative, unless  
   $v$ is a terminal vertex whose preimages all have valence 2   
(in this case $\Delta_v=-\frac12$). Note   that such a $v$   is a 2-plateau, since the label near $v$   must be even. Thus   $\Delta_v$ is always non-negative if $\Gamma$ contains no proper 2-plateau.

Also note that the second equality of Lemma \ref{simp} yields 
$$\displaystyle\sum_{v\in V}\Delta_v=(\ov \beta+\ov t)-(\beta+  t).$$ 

\begin{rem*}
Since    $\Delta_v$ may be negative, 
it is not always true that $
\beta +t\le \ov \beta+ {\ov t} $ (for instance, $\ov \Gamma$ may be a circle subdivided into two intervals,   mapping onto an interval). See Lemma \ref{bt} for a complete discussion. 
\end{rem*}

\begin{dfn}[boundary, frontier point, interior plateau] \label{bord}
The \emph{boundary} $\bo P$ of  a $p$-plateau is the set of oriented edges $e$ with origin $v$ in $P$ and  terminal point not in $P$;   the vertex $v$ is a \emph{frontier point} of $P$.  Note that the label $\lambda_e$ of the edge $e$ is divisible by $p$.   

A plateau $P $ is \emph{interior} if $P$ is not the whole graph and $P$  contains no terminal vertex.
\end{dfn}

Every terminal vertex of $\Gamma$ (or $\ov \Gamma$) is a plateau, so  $\mu$ is the sum of $t$ and the minimal cardinality of a set meeting every interior plateau (unless $t=0$ and $\Gamma$ is the only plateau).

\begin{dfn}[totally unfolded] A  plateau $P\inc\Gamma$ is \emph{$p$-totally unfolded} if $P$ is a $p$-plateau and, given any edge $e\in\bo P$  with origin $v\in P$, and $x\in\ov \Gamma_v$, the number of lifts of $e$ at $x$ is divisible by $p$.  A  plateau   is  \emph{totally unfolded} if it is $p$-totally unfolded for some $p$.   
\end{dfn}

 It may happen that $P$ is both a $p$-plateau and a $p'$-plateau, totally unfolded as a $p$-plateau but not as a $p'$-plateau.  We  consider such a $P$ as totally unfolded. 

The arguments  in the previous subsection show:
\begin{lem} \label{uf}
If no component of $\pi\m(P)$ is a plateau, then   $P$ is  totally unfolded. \qed
\end{lem}

\begin{dfn} [minimal plateau, $c$] \label{minim}
  A \emph{minimal plateau} is a plateau  $P\inc\Gamma$ which is interior, totally unfolded, and minimal for these properties. If   a minimal  plateau   is $p$-totally unfolded, we say that it is a minimal $p$-plateau.
  
  Let $c$ be the minimal cardinality of a set of vertices meeting every minimal plateau.
\end{dfn}

\begin{lem} \label{delt}
If $\Gamma$ contains no proper 2-plateau, one has $\beta +t+c\le \ov \beta +\ov t$.
\end{lem}

\begin{proof}  
We have seen that   
$\sum_{v\in V}\Delta_v=(\ov \beta +\ov t)-(\beta +  t),$
so it suffices to show  $\sum_{v\in V}\Delta_v\ge   c$. Since $\Delta_v$ is non-negative at terminal vertices (because there is no 2-plateau), we may restrict the sum to non-terminal vertices. 

For each minimal plateau $P_i$, we fix an edge $e_i\in \bo P_i$ with origin  
$v_i\in P_i$, and an odd prime  $p_i$ such that
$P_i$ is $p_i$-totally unfolded.

Given a non-terminal vertex $v$, consider 
  the minimal plateaux $P_i$ such that 
$v_i=v$. Let $n_v$ be their number (possibly 0). Note that the associated $p_i$'s are distinct. 
We shall show   $ \Delta_v\ge n_v$. Assuming this, $\sum_{v\in V}\Delta_v$ is bounded below by the number of minimal plateaux, hence by $c$ as required.

Fix $x\in \ov\Gamma_v$. Since $v$ is not terminal, the terms $\displaystyle\frac{d_x}2-1$ and $\displaystyle\frac{d_v}2-1$ are non-negative and 
$$2\Delta_v=d_x-d_v+\sum_{y\in\ov  \Gamma_v\setminus\{x\}}(d_y-2)\ge d_x-d_v.$$

Given an edge $e$ with origin $v$, consider the plateaux $P_j$ such that $e_j=e$.  
Let $n_e$ be their number (possibly 0). The number $d_{x,e}$ of lifts of $e$ at $x$ is divisible by the product of the corresponding $p_j$'s, so is bounded below by $2n_e+1$ (because the product of $n$ odd prime numbers is at least $2n+1$). 
 Summing over all edges $e$ with origin $v$ we get $$d_x= \sum_{e\in E_v} d_{x,e}\ge 2\sum_{e\in E_v} n_e+d_v= 2n_v+d_v,$$ so $2\Delta_v\ge d_x-d_v\ge2n_v$ as required.
  \end{proof}
  
  \begin{cor} \label{cassim}
  Proposition \ref{keyg} is true if $\Gamma$ contains no proper 2-plateau.
  \end{cor}
  
  \begin{proof}  One has $ \ov \mu= \ov t+ | A | $, where $A\inc\ov  \Gamma$ is a subset of minimal cardinality such that every plateau of $\ov\Gamma$ meets $A$ or contains a terminal vertex.
  If $C\inc V$ is a set of minimal cardinality $c$ meeting every minimal plateau of $\Gamma$, the union of  $C\cup\pi(A)$ with the terminal vertices meets every   plateau of $\Gamma$ (totally unfolded or not) by Lemma \ref{uf}, so $\mu\le c+ | A | +t$. Using $\beta +t+c\le\ov  \beta +\ov t$, we get $$\beta +\mu\le \beta +c+ | A | +t\le \ov \beta +\ov t+| A | = \ov \beta +\ov \mu.$$  
    \end{proof}
    
    Proposition \ref{keyg} thus follows fairly directly from the inequality of
Lemma \ref{delt}. Unfortunately, that inequality   does not always hold, and in the general case we will have to deduce the proposition from  the weaker lemma \ref{deltf}.

    \subsection{The general case}
Unless mentioned otherwise, $\pi$ is as in Proposition \ref{keyg}.

The assumption that there is no proper 2-plateau was used  in the previous subsection  to ensure $\Delta_v\ge0$  and $\prod_{j=1} ^{n_e}p_j\ge2n_e+1$. This motivates the following definitions. 

\begin{dfn}[bad vertex, good vertex, $V_g$] A   vertex $v$ of $\Gamma$ is \emph{bad}  if $\Delta_v<0$, i.e.\  if $v$ is terminal and all its preimages in $\ov \Gamma$ have valence 2. The label near a bad vertex is even.

Let $V_g$ be the set of \emph{good} (i.e.\ not bad) vertices of $\Gamma$.
\end{dfn} 

\begin{dfn}[bad plateau]
A minimal plateau $P\inc \Gamma$  is \emph{bad} if it is   2-totally unfolded, its boundary consists of a single edge $e$, and $e$ has exactly 2 lifts in $\ov \Gamma$ (so the origin of $e$ has a single  preimage since $P$ is totally unfolded).
\end{dfn}

\begin{dfn}  [$t_g$, $c_g$]
Let $t_g$ be the number of good terminal vertices, and $c_g$  
 the number of good minimal plateaux. 
\end{dfn} 

\begin{lem} \label{delt2}
 One has $\beta +t_g+c_g\le\ov  \beta +\ov t$.
\end{lem}

\begin{proof}

We argue as in the proof of Lemma \ref{delt}. Since $\Delta_v=-\frac12$ if $v$ is bad, we have $\sum_{v\in V\setminus V_g}\Delta_v=
-\frac12(t-t_g)$ and 
$$\sum_{v\in V_g}\Delta_v=\sum_{v\in V }\Delta_v+\frac12(t-t_g)=(\ov \beta +\ov t)-(\beta +  t)+\frac12(t-t_g)\le (\ov \beta +\ov t)-(\beta +  t_g),$$
and we  reduce to  showing $\sum_{v\in V_g}\Delta_v\ge c_g 
$.
We define $p_i$, $e_i$, $v_i$, $n_v$, $n_e$ as above (with $p_i=2$ now allowed), except that we restrict  to good plateaux, and we try to prove $\Delta_v\ge n_v$ for $v\in V_g$ (we will not succeed in all cases). 

In the proof of Lemma \ref{delt}, we deduced 
  $\Delta_v\ge n_v$ from the  inequalities 
       \begin{eqnarray*}  
           2\Delta_v&\ge& d_x-d_v
       \\		    
      d_x =\sum_{e\in E_v} d_{x,e}&\ge&\sum_{e\in E_v}\,(\prod   _{j=1}^{n_e}
      p_j)
       \\     
  \prod  _{j=1}^{n_e}p_j&\ge&2n_e+1
    . 
    \end{eqnarray*}
     
     The first two inequalities are still true, but the third  inequality 
may be wrong. This happens  precisely when $n_e=1$ and $p_1=2$ (recall that the $p_j$'s are distinct), so we have $\Delta_v\ge n _v$ unless $v$ is the chosen frontier point $v_i$ of a 2-totally unfolded plateau $P_i$. In  this case we have only proved $2\Delta_v\ge 2 n_v-1$ (this does not imply  $\Delta_v\ge n_v$ because $\Delta_v$ is not necessarily an integer). To conclude, we shall use goodness of $P_i$ to   find an edge in $\bo P_i$ with a lift not taken into account in the previous estimates. 

If $e$ has more than 2 lifts in $\ov \Gamma$, or if there are 2 edges with origin $v$ in $\bo P_i$,  
    we still have $\Delta_v\ge n_v$ because one of the first two  inequalities displayed above is strict.  Otherwise, since $P_i$ is good, it has another frontier point $w_i$. Every edge with origin $w_i$ in $\bo P_i$ has at least 2 lifts,  and we   have $\Delta_v+\Delta_{w_i}\ge n_v+n_{w_i}$ even if $\Delta_v< n_v$.   The $w_i$'s are distinct because 2-plateaux are equal or disjoint, so we get $\sum_{v\in V_g}\Delta_v\ge \sum_{v\in V_g}n_v=c_g$.
 \end{proof}
   
n the following definitions, and in Lemma \ref{bt}, the graphs do not have to be labelled, and $\pi$ is just an open morphism between finite connected graphs. 

\begin{dfn}[accordion] \label{acco}  
$\pi:\ov \Gamma\to\Gamma$ is an \emph{accordion}  if $\ov\Gamma$ is homeomorphic to a circle and $\Gamma$ to an interval (see Figure \ref{aco}). Thus the terminal vertices of $\Gamma$ are bad, and $\ov \Gamma$ is the union  of $2n$ intervals (possibly subdivided), each mapped homeomorphically to $\Gamma$. We call $n$ the  \emph{size} of the accordion.
\end{dfn}

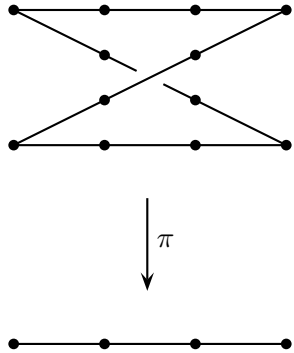
\begin{figure}[h]
\begin{center}
\begin{pspicture}(100,130) 
       \rput (0,75){ 
 \psline  (0,0)(102,0)
 \psline  (0,51)(102,51)
  \psline  (0,0)(102,51)
\psline  (0,51)(46,28)
 \psline  (56,23)(102,0) 
  \pscircle*( 0,0) {2}
   \pscircle*( 0,51) {2} \pscircle*( 102,51) {2} \pscircle*( 102,0) {2}   
   \pscircle*( 34,0) {2}
   \pscircle*( 68,0) {2} \pscircle*( 34,51) {2} \pscircle*( 68,51) {2}   
   \pscircle*( 34,34) {2}
   \pscircle*( 68,34) {2} \pscircle*( 34,17) {2} \pscircle*( 68,17) {2}
 }
\psline [arrowsize=5pt] {->}(50,55)(50,20)  
 \rput(57,39){$\pi$} 
 \psline  (0,0)(102,0) 
   \pscircle*( 0,0) {2}
     \pscircle*( 34,0) {2}  \pscircle*( 68,0) {2}  \pscircle*( 102,0) {2}\end{pspicture}
\end{center}
\caption{an accordion of size 2
} \label{aco}
\end{figure}

\begin{dfn} [branched covering] \label{bc}  
$\pi:\ov \Gamma\to\Gamma$ is a \emph{branched 2-covering of a tree} (or simply a branched covering) if $\Gamma$ is a tree, and for $u\in \Gamma$ the preimage $\ov \Gamma_u$ consists of a single  point  if $u$ is   a terminal vertex of $\Gamma$, of 2 points otherwise (see Figure \ref{bct}). In particular, all terminal vertices of $\Gamma$ are bad. 
\end{dfn}

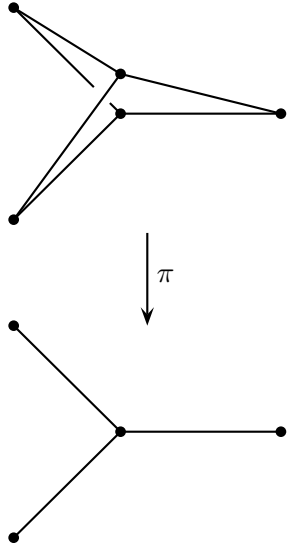
\begin{figure}[h]
\begin{center}
\begin{pspicture}(100,205) 
   \psline  (0,0)(40,40)   
   \psline  (0,80)(40,40)      
    \psline  (40,40)(100,40)  
  \pscircle*( 0,0) {2}
   \pscircle*( 40,40) {2} \pscircle*( 100,40) {2} \pscircle*( 0,80) {2}
   \rput (0,120){ 
     \psline  (0,0)(40,40) 
     \psline  (0,80)(30,50) 
   \psline  (36,44)(40,40)         
    \psline  (40,40)(100,40)    
     \psline  (0,0)(40,55) 
   \psline  (0,80)(40,55) 
    \psline  (40,55)(100,40)      
 \pscircle*( 0,0) {2}
   \pscircle*( 40,40) {2} \pscircle*( 100,40) {2} \pscircle*( 0,80) {2}
     \pscircle*( 40,55) {2}
    }
   \rput (0,60){ 
\psline[arrowsize=5pt]  {->}(50,55)(50,20)
 \rput(57,39){$\pi$}
 }
\end{pspicture}
\end{center}
\caption{a branched 2-covering of a tree
} \label{bct}
\end{figure}

An accordion of size 1 is a branched  covering.

The following topological lemma may be seen as a warm-up for the proof of Lemma \ref{deltf}.

\begin{lem} \label{bt} Let $\pi:\ov \Gamma\to\Gamma$ be an open morphism between finite connected graphs.  
One has $\beta +t=\ov  \beta +\ov t+1$ if $\pi$ is   an accordion or a branched 2-covering of a tree, $\beta +t\le\ov  \beta +\ov t$ otherwise.
\end{lem}

\begin{proof} 
If $\pi$ is an accordion, one has $\beta +t=0+2=2$  and $\ov \beta +\ov t=1+0=1$.
If it  is a branched 2-covering of a tree, then $\beta =\ov t=0$ and $\ov \beta =t-1$. We now consider the general case. 

As a preliminary observation, note that a connected graph satisfies $\beta +t\ge2$ unless it is homeomorphic to a point or a circle. If $\ov\Gamma$ is a circle, $\Gamma$ is a circle or an interval. Also note that the lemma is true if $\Gamma$ has only one edge. We will argue by induction on the number of edges of $\Gamma$. 

Let $v$ be a bad vertex of $\Gamma$ (the result follows from Lemma \ref{delt2} if there is none). We define a vertex 
$w$ as follows (see Figure \ref{dw}). 

\begin{figure}[h]
\begin{center}
\begin{pspicture}(100,140) 
 \psline  (0,0)(140,0)
  \pscircle*( 0,0) {2}
   \pscircle*( 60,0) {2} \pscircle*( 120,0) {2}  
   \rput( 0,10){$v$}
     \rput( 60,10){$v_1$}
       \rput( 120,10){$w$}
\psline[arrowsize=5pt]  {->}(70,55)(70,30)
   \rput(77,43){$\pi$}     
 \rput(0,100){ 
 \psline  (0,0)(60,20)
 \psline  (0,0)(60,-20)
  \psline  (60,-20)(120,-10)
\psline  (60,-20)(120,-30)
\psline  (60,20)(120,20)
 \psline  (120,20)(140,25)
 \psline  (120,20)(140,15)
  \psline  (120,-10)(140,-5)
\psline  (120,-10)(140,-15)
\psline  (120,-30)(140,-35)
 \psline  (120,-30)(140,-25)
\rput(60,30){$x_0$}
  \pscircle*( 0,0) {2}
   \pscircle*( 60,20) {2} \pscircle*( 60,-20) {2}  
     \pscircle*( 120,20) {2} \pscircle*( 120,-10) {2}   \pscircle*( 120,-30) {2}   
 }  
\end{pspicture}
\end{center}
\setlength\belowcaptionskip{-.5cm}
\caption{the vertex $w$ ($w$ cannot  be $v_1$ because of $x_0$)
} \label{dw}
\end{figure}
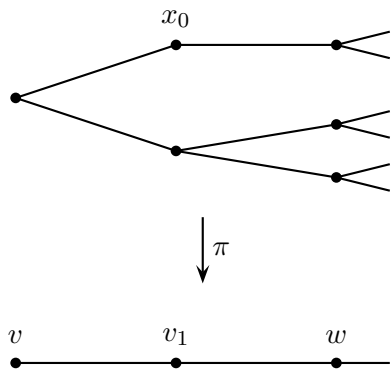

Let $v_1$ be the vertex adjacent to the terminal vertex $v$. 
We   let $w=v_1$ if  the valence of $v_1$ is different from 2, or if  the valence is 2 and every $x\in\ov \Gamma_{v_1}$ is the origin of at least 2 edges not mapping onto $v_1v$. Otherwise, we consider the vertex $v_2\ne v$ adjacent to $v_1$ and we iterate. We obtain a vertex $w=v_q$ such that $w$ has valence $\ne2$, or $w$ has valence 2 and every $x\in\ov \Gamma_{w}$ is the origin of at least 2 edges not mapping onto $wv_{q-1}$.

If  $w $  is  terminal, then   $\Gamma$ is a segment,    $\beta +t=2$, and $\ov \beta +\ov t\ge2$ since  $\pi$ is not an accordion, so the lemma is proved. 

Otherwise, we consider 
 the graph $\Gamma'$ obtained from $\Gamma$ by removing the segment $vw$ (but not the vertex $w$). No vertex of  $\ov \Gamma_{w}$ is terminal in $\ov \Gamma'=\pi\m(\Gamma')$. With obvious notations, one has $\beta '=\beta $ and $t'\ge t-1$, with equality if and only if $w$ has valence $\ge3$ in $\Gamma$.  
We distinguish two cases. 
  
  $\bullet$ First case: $\ov \Gamma'$ is connected. We then consider  $\ov \beta '$ and $\ov t'$. The second equality of Lemma \ref{simp} implies $\ov \beta '+\ov t'<\ov \beta +\ov t$,  because any $x\in\ov\Gamma_w$ has valence $\ge2$ in $\ov \Gamma'$, so $ | \frac{d_x}2-1 | $ is larger when computed in $\ov \Gamma$ than when computed in $\ov \Gamma'$. 
  
  The restriction   $\pi':\ov \Gamma'\to\Gamma'$ of $\pi$ is open,
  so we may use  induction  
  and write 
  $$
 \beta +t\le \beta '+t'+1\le \ov \beta '+\ov t'+2\le \ov \beta +\ov t+1.
$$

The lemma is proved if one of the inequalities is strict. If not,   $w$ has valence $\ge3$ in $\Gamma$, the map $\pi'$ is an accordion or a branched covering, and $\ov \beta '+\ov t'=\ov \beta +\ov t-1$. 

This equality drastically limits the possibilities for the preimage of the segment $vw$ (see Figure \ref{preim}): either $w$   has two preimages $x,x'$, and $\pi\m(vw)$ is an arc joining them, or $w$ has a single preimage $x$, and  
$\pi\m(vw)$ is a circle containing $x$ or a lollipop (an arc $xy$ with a circle attached to $y$).

\begin{figure}[h]
\begin{center}
\begin{pspicture}(100,140)  
\psline[arrowsize=5pt] {->}(50,55)(50,25)
  \rput(57,41){$\pi$}
 \psline  (0,0)(100,0) 
   \pscircle*( 0,0) {2}
     \pscircle*( 100,0) {2}    \rput( 0,10){$v$}
            \rput( 100,10){$w$}
      \rput (-150,100){  
     \psline  (0,0)(100,20) 
          \psline  (0,0)(100,-20) 
          \rput(110,20){$x$}
             \rput(112,-18){$x'$}
   \pscircle*( 0,0) {2}
     \pscircle*( 100,20) {2}   \pscircle*( 100,-20) {2}  
}
    \rput (50,100){  
    \psellipse(0,0)(50,20)
           \rput(60, 0){$x$}
     \pscircle*( 50,0) {2}   
}
    \rput (200,100){  
    \psellipse(-25,0)(25,20)
    \psline(0,0)(50,0)
           \rput(60, 0){$x$}
             \rput(7,8){$y$}
     \pscircle*( 50,0) {2}    
}
\end{pspicture}
\end{center}
\caption{possible preimages of $vw$
} \label{preim}
\end{figure}
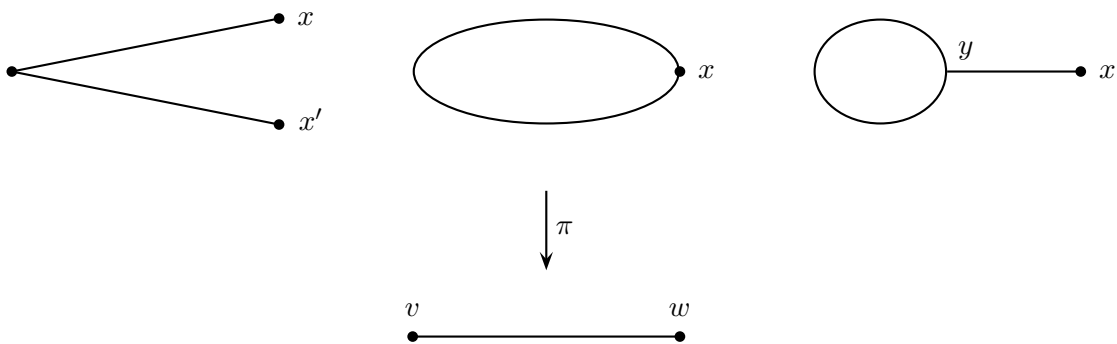

We complete the proof by showing that $\pi$ must be  a branched covering.
Since $\pi'$ is an accordion or a branched covering, $w$ has at least 2 preimages  because it is not terminal in $\Gamma'$. 
It follows that  $w$   has two preimages, and $\pi\m(vw)$ is an arc joining them, so $\pi$ is  a branched covering  of the tree $\Gamma=\Gamma'\cup vw$.

$\bullet$ Second case: $\ov \Gamma'$ has several components $\ov \Gamma'_i$. By induction, $\ov \beta '_i+\ov t'_i\ge \beta '+t'-1$. Also, $ \ov \beta +\ov t\ge\sum_i(\ov \beta '_i+\ov t'_i)$ since no vertex in $\ov\Gamma_w$ is terminal in $\ov\Gamma' $.
We then write
$$\ov \beta +\ov t\ge\sum_i(\ov \beta '_i+\ov t'_i)\ge2(\beta '+t'-1)\ge2(\beta +t-2)=\beta +t+(\beta +t-4).
$$
 
 We are done if $\beta +t\ge4$. Since $\Gamma$ is not a circle, one cannot have $\beta +t=1$. If $\beta +t=2$, one has $\ov \beta +\ov t\ge2$ since $\ov \Gamma$ is not a circle. There remains the possibility that $\beta +t=3$ and $\ov \beta +\ov t=2$. If this happens, one has $\sum_i(\ov \beta '_i+\ov t'_i)=2$, so $\ov \Gamma'$ consists of two disjoint circles and $\ov \Gamma$ is their union with an arc. It follows that $\Gamma$ is homeomorphic to an arc or a lollipop, 
 so $\beta +t=2$.
\end{proof}

\begin{lem} \label{par} Let $\pi:\ov \Gamma\to\Gamma$ be as in Proposition \ref{keyg}. If $P$ is a 2-plateau, the parity of the cardinality of $\pi\m(u)$ is the same for all $u\in P$.
If $P$ is a 2-totally unfolded plateau, and   $P\ne\Gamma$, all points in $\pi\m(P)$ have even multiplicity; in particular, the total multiplicity of $\pi$ is even.
\end{lem}

\begin{proof} This follows easily from condition $(*)$. With the notations of Lemma \ref{loc}, if $e\inc P$, then $\lambda_e$ is odd, so $k _{x,e}$ is odd. This proves the first assertion. 

If $P$ is 2-totally unfolded, points of $\ov \Gamma$ mapping to a frontier point  of $P$ have even multiplicity (because $m_x$ is even if $k _{x,e}$ is), and this propagates to all points    of $\pi\m(P)$ since all numbers $k _{x,e}$ associated to edges   in $P$ are odd. 
\end{proof}

\begin{dfn} [generalized branched covering]\label{gbc}
 We say that $\pi$ is a \emph{generalized branched 2-covering of a tree} if   
there exist   minimal 2-plateaux $P_i$  such that: 
\begin{enumerate}
\item the graph obtained from $\Gamma$ by collapsing each $P_i$ to a point is a tree;
\item if $u\in \Gamma$ is not a terminal vertex and does not belong to $\cup_i P_i$, its preimage $\ov \Gamma_u$ consists of 2 points;
\item if $v$ is a terminal vertex of $\Gamma$, or a frontier point of some $P_i$, its preimage is a single point (in particular, terminal vertices are bad). 
\end{enumerate}
The $P_i$'s are the \emph{branching plateaux} of $\pi$. The union of the $P_i$'s and the terminal vertices is the \emph{branching locus}. 
\end{dfn}

See a simple  example on Figure \ref{disp}. The ellipse is the unique branching plateau.

When there is no $P_i$, we recover the definition of a branched covering (Definition \ref{bc}).

Recall (Definition \ref{minim}) that a minimal 2-plateau is interior, $2$-totally unfolded, and minimal among all interior totally unfolded plateaux (not necessarily 2-plateaux). Any two 2-plateaux are equal or disjoint, so the branching plateaux are disjoint.
It follows from Lemma \ref{par}  that all points in a branching plateau have an odd number of preimages, so the branching locus is the set of points with an odd number of preimages.  

\begin{dfn} [exceptional] \label{excep} 
$\pi$ is \emph{exceptional} if it is an accordion or a generalized branched 2-covering of a tree.
\end{dfn}

\begin{lem} \label{deltf}
If $\pi:\ov \Gamma\to\Gamma$ is as in Proposition \ref{keyg}, then:
$$ \beta +t+c\le\ov  \beta +\ov t \mbox{\quad if $\pi$ is not exceptional}$$
$$ \beta +t+c\le\ov  \beta +\ov t +1 \mbox{\quad if $\pi$ is   exceptional.}$$
\end{lem}

Recall that $c$ is the minimal cardinality of a set of vertices meeting every minimal plateau.

\begin{proof} We elaborate on the proof of Lemma \ref{bt}.  
If $\Gamma'$ is a connected subgraph, the restriction $\pi'$ of $\pi$ to the preimage $\ov \Gamma'$ 
 satisfies $(*)$, and $\Gamma'$ is reduced, so we can argue inductively,  defining minimal plateaux of $\Gamma'$, and a number $c'$,  as in Subsection \ref{sim} (Definition \ref{minim}).  

The lemma follows from Lemma \ref{bt} if $c=0$, in particular if $\Gamma$ has  only one edge or if $\pi$ is an accordion. It follows from Lemma \ref{delt2} if there is no bad vertex or bad plateau.

$\bullet$ First suppose that there is a bad vertex $v$. We define $w$, $\Gamma'$,   $\ov \Gamma'$ as in the proof of Lemma \ref{bt}. 

 It follows from the definition of $w$ that any interior  totally unfolded plateau which meets the segment $vw$ must contain $w$. In particular $c=0$ if $w$ is terminal, so the result is true in this case. We therefore assume that $w$ is not terminal (in particular, $\ov \Gamma$ is not a circle).

 The intersection of a minimal plateau of $\Gamma$ with $\Gamma'$ is totally unfolded, but it is not necessarily interior: it may be equal to $\Gamma'$, and $w$ may be a terminal vertex of $\Gamma'$  (if it has valence 2 in $\Gamma$). 
 
 If $A\inc \Gamma'$ meets every minimal plateau of $\Gamma'$, then $A\cup\{w\}$ meets every minimal plateau of $\Gamma$, so $c\le c'+1$. Equality is possible only if $\Gamma'$ is contained in  a minimal plateau of $\Gamma$, or $w$ has valence 2 in $\Gamma$.

1) First assume that no minimal plateau of $\Gamma$ contains $\Gamma'$. Then 
$$ \beta +t+c\le \beta '+t'+c'+1$$ since $t\le t'+1$ with equality only if $w$ has valence $>2$ in $\Gamma$ (see the proof of Lemma \ref{bt}), and $c\le c'+1$ with equality only if $w$ has valence 2 (and of course $\beta =\beta '$). We now argue as in the proof of Lemma \ref{bt}, distinguishing two subcases.  
  
1a)  If $\ov \Gamma'$ is connected, we have 
$$\beta +t+c\le \beta '+t'+c'+1\le \ov \beta '+\ov t'+1\le\ov  \beta +\ov t$$ if $\pi'$ is not exceptional (as above, the last inequality follows from Lemma \ref{simp}). If $\pi'$ is   exceptional,  we only get $\beta +t+c\le  \ov \beta +\ov t+1$. We assume that equality holds, and we show that $\pi$ is exceptional or $c=0$. 

We have    $\ov \beta '+\ov t'=\ov \beta +\ov t-1$, so $\ov \Gamma\setminus\ov \Gamma'$ is very restricted (as pointed out in the proof of Lemma \ref{bt}, $\pi\m(vw)$ is an  arc, a circle, or a lollipop). In particular,  $w$ has at most 2 preimages; if $\pi'$ is an accordion, it has size (in the sense of Definition \ref{acco}) at most 2. 

If $w$ has   2 preimages, then $\pi$ is exceptional  with the same $P_i$'s as $\pi'$ if $\pi'$ is a generalized branched covering, while $c=0$ if $\pi'$ is an accordion of size 2. If $w$ has 1 preimage, let $w_1w$ be the maximal subsegment of $vw$ consisting of points with a single preimage ($w_1=\pi(y)$ if $y$ is as on Figure \ref{preim}). Points between $v$ and $w_1$ have 2 preimages  ($w_1\ne v$ since $v$ is a bad vertex).

Define the branching locus of $\pi$ as the union of that of $\pi'$ with $w_1w$ and $\{v\}$.
The component containing $w$ is a   2-totally unfolded plateau of $\Gamma$ because, if $e$ is an oriented edge with origin in $w_1w$, its label $\lambda_e$ has the same parity as the number of lifts of $e$ (the total multiplicity of $\pi$ is even by Lemma \ref{par}). It is minimal because the branching plateaux of $\pi'$ are. 
The conditions of Definition \ref{gbc} are satisfied, so $\pi$ is exceptional.  

1b) If $\ov \Gamma'$ has several components $\ov \Gamma'_i$,
then
$$\ov \beta +\ov t\ge\sum_i (\ov \beta '_i+\ov t'_i)\ge 2(\beta '+t'+c'-1)\ge 2(\beta +t+c-2)= \beta +t+c+(\beta +t+c-4 ).$$
As in the proof of Lemma \ref{bt}, we just have to rule out the possibility that  $\beta +t+c=3$ and $\ov \beta +\ov t=2$. If this happens, $\ov \Gamma'$ is the union of two disjoint circles and an arc, mapping onto an arc or a lollipop, and $c=0$.  

2) If some minimal plateau $P$ of $\Gamma$ contains $\Gamma'$, we have $t=1$ (the only terminal vertex of $\Gamma$ is $v$), $c=1$ (every minimal plateau contains $w$), and $\ov t=\ov t'=0$.

2a) If $\ov \Gamma'$ is connected, we have 
$$\beta +c+t=\beta '+2\le \ov \beta '+2\le \ov \beta   +1.$$
We assume that both inequalities are equalities, and we show that $\pi$ is exceptional. 

The equality $\ov \beta '+2= \ov \beta   +1$ implies that one gets $\ov \Gamma$ from $\ov \Gamma'$ by attaching an arc, a circle,  or a lollipop. It cannot be an arc because of the totally unfolded interior plateau $P$ containing $\Gamma'$. It follows that $w$ has a single preimage.

If $\pi'$ is   exceptional, one shows that $\pi$ is exceptional as in case 1a.
If not, we have $\beta '+t'\le \ov \beta '+\ov t'= \ov \beta '$ by Lemma \ref{bt}, so $t'=0$ since we assume  $\beta '=\ov \beta '$. Thus $\Gamma'$ and $\ov \Gamma'$ are graphs with no terminal vertices, and $\beta '=\ov \beta '$ implies that $\pi'$ is an isomorphism by   Lemma \ref{simp}. It follows that $\pi$ is exceptional, with branching locus $\Gamma'\cup w_1w\cup\{v\}$.

2b)  If $\ov \Gamma'$ is not connected, we have 
$$ \beta +c+t-1=  \beta '+1\le2\beta '\le \ov \beta '_1+\ov t'_1+\ov \beta '_2 +\ov t'_2\le \ov \beta +\ov t.
$$

The first inequality holds because $\beta '\ge1$: since $\Gamma'$ is contained in  a minimal plateau of $\Gamma$, no vertex other than $w$ may be   terminal in  $\Gamma'$. The second inequality comes from Lemma \ref{simp2}. The third inequality was used in the proof of Lemma \ref{bt};  as in the previous subcase, the existence of $P$ implies that $\ov \Gamma\setminus\ov  \Gamma'$ cannot be an open segment, so the inequality is strict. We deduce $ \beta +c+t \le\ov \beta +\ov t$.

$\bullet$ We now assume that there is a bad plateau $P$. Let $v$ be its unique frontier point. Let $\Gamma'$ be the graph obtained from $\Gamma$ by removing $P\setminus\{v\}$. The preimages $\ov \Gamma'$ of $\Gamma'$ and $\ov P$ of $P$ are connected because $\ov\Gamma_v$ is a single point. 

Denoting by $\beta (P)$ and $\beta (\ov P)$ the first Betti numbers of $P$ and its preimage $\ov P$, we have $\beta (\ov P)\ge \beta (P)$ by Lemma \ref{simp2} and 
 \begin{eqnarray*} 
 t'&=&t+1 \\
\ov  t'&=&\ov t \\
 \beta '&=&\beta -\beta (P) \\
 \ov \beta '&=&\ov \beta -\beta (\ov P).
  \end{eqnarray*} 

By minimality of $P$, any minimal plateau of $\Gamma$ either contains $v$, or is contained in $\Gamma'\setminus \{v\}$ and is a minimal plateau of $\Gamma'$. This shows  $c\le c'+1$.
By induction we get 
 $$\beta +t+c\le \beta '+\beta (P) +t'-1+c'+1\le \beta '+\beta (\ov P)+t'+c'\le \ov \beta '+\ov t'+\beta (\ov P)=\ov \beta +\ov t
$$
if  $\pi'$  is not  exceptional. If $\pi'$ is exceptional, so is $\pi$ (with $P$ added to the branching locus), and we get $\beta +t+c\le \ov  \beta +\ov t+1$ as required.
 \end{proof}
 
 We can now conclude.
 
 \begin{proof}[Proof of Proposition \ref
 {keyg}]
 
 If $\pi$ is not exceptional, we write
 $$\beta +\mu\le \beta +c+ | A | +t\le \ov \beta +\ov t+| A | = \ov \beta +\ov \mu$$ 
 as in the proof of Corollary \ref{cassim}, with $A\inc \ov\Gamma$   a subset of minimal cardinality 
 such that every plateau of $\ov \Gamma$ meets $A$ or contains a terminal vertex.
 
 If $\pi$ is exceptional, we only have $\beta +c+t\le \ov \beta +\ov t+1$. In this case we get the required inequality $\beta +\mu\le \ov \beta +\ov \mu$ by   showing $\mu<c+ | A | +t$. We fix $x\in A$, noting that $A$ is nonempty because $\ov \Gamma$ has no terminal vertices.  
 
 If $\pi$ is an accordion, we claim that $\pi(A\setminus\{x\})$ meets every interior plateau $P\inc \Gamma$. This implies $\mu< | A | +t$, hence the result since $c=0$. To prove the claim, consider the preimage of $P$. It  consists of $2n$ disjoint arcs, each a plateau of $\ov \Gamma$ (with $n\ge1$ as in Definition \ref{acco}). Each of these plateaux meets $A$, so one of them contains a point of $A\setminus\{x\}$. 
 
 In the case of a generalized branched covering, $c$ is bounded below by $k$, defined as the number of branching plateaux $P_i$. We prove $\mu<c+ | A | +t$ by constructing a set $C\inc \Gamma$ of cardinality at most $k+ | A | +t-1$ meeting every plateau.
  There are two cases. 
  
  First suppose that  $\pi(x)$ belongs to a branching plateau $P_{i_0}$. For $i\ne i_0$, let $v_i$ be the point of $P_i$ closest to $P_{i_0}$; it is well-defined because one gets a tree by collapsing each $P_i$  to a point. Let $C$ consist of the $v_i$'s, the terminal vertices, and $\pi(A)$. We show that $C$ meets every interior plateau $P\inc \Gamma$.
   
  If $P$ is not totally unfolded, it meets $\pi(A)$ by Lemma \ref{uf}, so assume that it is.  If it is contained in some $P_i$, it equals $P_i$ by minimality of $P_i$, so contains $\pi(x)$ or   $v_i$. Otherwise, because of condition 2 in Definition \ref{gbc}, $P$ contains the closure of a component of $\Gamma\setminus\cup_i P_i$. Since one gets a tree by collapsing each $P_i$  to a point, this closure contains a $v_i$ or a terminal vertex. 
 
 Now suppose that $\pi(x)$ belongs to no $P_i$. For each $i$ we let $v_i$ be the point of $P_i$ closest to $\pi(x)$, and we let $C$ consist of the $v_i$'s, the terminal vertices, and $\pi(A\setminus\{x\})$.
 
 Arguing as in the previous case, we see that $C$ meets every totally unfolded interior plateau. We just have to check that it meets interior plateaux $P$ contained in the component of $\Gamma\setminus\cup_i P_i$ containing $\pi(x)$. As in the case of an accordion, $\pi\m(P)$ has two components, each a plateau of $\ov \Gamma$.  At least one of them does not contain $x$, so $P$ meets  $\pi(A\setminus\{x\})$.
  \end{proof}
  
  \section{Plateaux and finite index subgroups} \label{pfi}
  
  Let $\Gamma$ be a labelled graph representing a GBS group $G$. 
  We have seen in Subsection \ref{reduc} that any finite index subgroup $\ov G $ yields a map $\pi:\ov\Gamma\to\Gamma$ between labelled graphs. We formalize the properties of $\pi$.
  
    \begin{dfn}[Admissible map] Let $\ov\Gamma$ and $\Gamma$ be     labelled graphs (not necessarily connected). An \emph{admissible map} from $\ov\Gamma$ to $\Gamma$   is a pair $(\pi,m)$ where $\pi:\ov\Gamma\to\Gamma$   is a morphism, and $m$  assigns a positive multiplicity   to each vertex and edge of $\ov\Gamma$ so that condition $(*)$ of Lemma \ref{loc} is satisfied. We usually denote an admissible map simply by $\pi$, keeping $m$ implicit. 
    
   If $\Gamma$ is connected, the \emph{total multiplicity} of $\pi$ is  
$\sum_{x\in\pi\m(u)}m_x$; it does not depend on $u\in \Gamma$.    \end{dfn}

Recall   condition $(*)$:  if $\pi(x)=o(e)$, then  $e$ has $k_{x,e}= m_x\wedge \lambda_e $ lifts with origin $x$, each with multiplicity $m_x/k_{x,e}$ and label $\lambda_e/k_{x,e}$.

  \begin{rem}
Suppose $\ov\Gamma$ is connected. Given an admissible $\pi:\ov \Gamma\to\Gamma$, there are infinitely many $m$'s such that $(\pi,m)$ is admissible. One may show that  all of them are multiples of a single $m_0$ for which no prime divides all edge   multiplicities.  
\end{rem}

If $\pi:\ov\Gamma\to\Gamma$ and $\pi':\Gamma\to \Gamma'$ are admissible, so is $\pi'\circ \pi$, with multiplicity function $m (m'\circ \pi) $;  
this is easy to check, using the formula $ \lambda \wedge \mu\mu' =(\lambda\wedge \mu')(\frac{\lambda}{ \lambda\wedge \mu' }\wedge \mu)$.
  
  \begin{lem}   \label{adm}
  Let $\Gamma$ be a labelled graph representing a GBS group $G$. 
\begin{enumerate}
\item
If $\ov G$ is a subgroup of finite index, there is a labelled graph $\ov \Gamma$ representing $\ov G$ and an admissible map $\pi:\ov\Gamma\to\Gamma$ whose total multiplicity  is the index of $\ov G$.
 
\item Conversely, any admissible map $\pi:\ov\Gamma\to\Gamma$ with $\ov\Gamma$ connected may be obtained from a finite index subgroup $\ov G$ as in 1. In particular, the group represented by $\ov \Gamma$ embeds into $G$ as a finite index subgroup.
\end{enumerate}
\end{lem}  
  
\begin{proof}
The first assertion was proved in Subsection \ref{reduc}.  The second assertion may   be deduced from \cite{Ba}, but we provide a direct proof. 
We consider the 2-complex $\Theta$ associated to $\Gamma$ as in Subsection \ref{reduc} and we construct a covering map $\rho:\ov \Theta\to \Theta$ inducing $\pi$. The desired group $\ov G$ is the fundamental group of $\ov \Theta$.

Let $C=\R/\Z$ be the standard circle. We view $\Theta$ as the union of circles $C_v=\{v\}\times C$ and annuli $A_\varepsilon=\varepsilon\times C$. The complex $\ov \Theta$ is made of circles $C_x=\{x\}\times C$ associated to vertices of $\ov\Gamma$ and annuli $A_{\ov\varepsilon}=\ov\varepsilon\times C$ associated to edges. If $x$ is an endpoint of $\ov\varepsilon$, we attach the corresponding boundary circle of $A_{\ov\varepsilon}$ to $C_x$ by the map $(x,\theta)\mapsto (x,\lambda\theta)$, where $\lambda$ is the label carried by $\ov\varepsilon$ near $x$. To define $\rho$, we map each circle $\{y\}\times C$, for $y\in\ov\Gamma$, to $\{\pi(y)\}\times C$ by $(y,\theta)\mapsto (\pi(y),m_y\theta)$. Condition $(*)$ ensures that this is compatible with the attaching maps of $\Theta$ and $\ov \Theta$. 
\end{proof}  
  
If $\pi: \Gamma'\to\Gamma$ is a covering map between finite graphs (in the topological sense), and $\Gamma$ is labelled, then labelling $ \Gamma'$ so that $\pi$ is label-preserving, and letting $m$ be  constant (coprime with all labels) on each component of $  \Gamma'$, makes $\pi$ admissible. 
Such an admissible map may be characterized as follows.

  \begin{lem} \label{cove}
  Let $\pi:\ov\Gamma\to\Gamma$ be an admissible map between labelled graphs. The following conditions are equivalent:
  \begin{enumerate}
\item $\pi$ is a covering map (in the topological sense);
\item all numbers $k_{x,e}= m_x\wedge\lambda_e $ are equal to 1;
\item $\pi$ preserves labels: $\lambda_{\pi(\ov e)}=\lambda_{\ov e}$;
\item the multiplicity $m$ is constant on every component of $\ov\Gamma$.
\end{enumerate}
 \end{lem}
  
If these conditions are satisfied, we say that   $\pi$ is a \emph{topological covering}, or that $\ov \Gamma$ is a topological covering of $\Gamma$. Note that every admissible $\pi$   is a covering of graphs of groups in the sense of \cite{Ba}.  

  \begin{proof} By $(*)$,   condition 2 is equivalent to $\pi$ being a covering, to $\pi$ being label-preserving, and also to every edge of $\ov\Gamma$ having the same multiplicity as its endpoints.
\end{proof}

 \begin{prop} \label{revet}
 Given a connected  labelled graph $\Gamma$, the   following conditions are equivalent:
   \begin{itemize}
\item every admissible $\pi:\ov\Gamma\to\Gamma$ is a topological covering;
\item $\Gamma$ contains no proper plateau. 
 \end{itemize}

\end{prop}
\begin{proof}

$\bullet$  We first suppose that $\pi$ is not a covering,  and we construct a proper plateau as in the proof of Proposition \ref{menage}.   Using Lemma \ref{cove}, fix a prime $p$ dividing some $k_{x_0 ,e_0 }$. It divides       $m_{x_0}$, so let $p^ \delta$ with $\delta\ge1$ be the maximal power of $p$ dividing the multiplicity of a vertex. 

We define a subgraph $\Gamma_1\inc \Gamma$ as follows. A vertex is in $\Gamma_1$ if and only if it has a preimage whose multiplicity is divisible by $p^ \delta$; an edge $\varepsilon$ is in $\Gamma_1$ if its endpoints are,  and none of the two labels carried by $\varepsilon$ is divisible by $p$. Note that $\Gamma_1\ne\Gamma$ because  $p$ divides $\lambda_{e_0}$. We show that every component of $\Gamma_1$ is a $p$-plateau. 

Consider an edge $e$ with origin $v\in\Gamma_1$ such that $p$ does not divide $\lambda _e$. We have to check that its other endpoint $w$ is in $\Gamma_1$, and the label  
near $w$ is not divisible by $p$.
 Let $x$ be a preimage of $v$ with multiplicity divisible by $p^\delta$, and $\ov e$ a lift of $e$ with origin $x$. Since $p$ does not divide $\lambda_{  e}$, the multiplicity of $\ov e$ is divisible by $p^\delta$. So is  the multiplicity of the terminal point $y$ of $\ov e$. In particular, $w\in \Gamma_1$. By our choice of $\delta$, the multiplicity of $y$ cannot be divisible by $p^{\delta+1}$. This implies that the label near $w$ 
 is not divisible by $p$. 

  $\bullet$ Conversely, given a proper $p$-plateau $P$, we construct a \emph{branched covering ramified over $P$}. It is an admissible map $\pi:\ov\Gamma\to\Gamma$ which is not a topological covering: points in $P$ have one preimage, points not in $P$ have $p$ preimages. Compare Figure \ref{disp}, which represents a covering ramified over the union of two 2-plateaux (the terminal vertex and the ellipse).
  
  To construct $\ov\Gamma$, we start with $\Gamma\times \{1,\dots,p\}$, and we identify $(x,i)$ with $(x,j)$ whenever $x\in P$. The map from $\ov\Gamma$ to $\Gamma$ is the natural projection. The multiplicity of (the image of) $(x,i)$ is defined as 1 if $x\notin P$, as $p$ if $x\in P$. The label of an edge $\ov e$ of $\ov\Gamma$ is the same as the label of its projection $e$, except if the origin of $e$  is in $P$ but $e$ is not contained in $P$; in this case $\lambda_e$ is divisible by $p$ and we define $\lambda_{\ov e} =\lambda_e/p$. One checks that $(*)$ is satisfied.
 \end{proof}
  
\begin{cor} \label{cov}
   Let $\Gamma$ be a labelled graph representing a GBS group $G$. Suppose that $\Gamma$ contains no proper plateau. Then a GBS group $\ov G$ is isomorphic to a finite index subgroup of $G$ if and only if it may be represented by a labelled graph $\ov \Gamma$ which is a topological covering of $\Gamma$. \qed
\end{cor}

 Recall that a group is \emph{large} if some finite index subgroup maps onto the free group $F_2$. A    GBS group represented by a labelled graph $\Gamma$ maps onto $F_2$ if and only if $\beta (\Gamma)\ge2$ (see \cite{Le}, p.\ 483). 

 \begin{thm} \label{larg}
 Let $G$ be a non-cyclic GBS group. The following are equivalent:
 \begin{enumerate}
\item $G$ is not large;
\item  $G$    may be represented by a labelled graph $\Gamma$   homeomorphic to a circle and containing  no proper plateau;
\end{enumerate}
\end{thm}

This  
was proved in \cite{EP} for $G=BS(m,n)$ (in this case the absence of a proper plateau is equivalent to $ m\wedge n =1$), and independently by T.\ Mecham \cite{Me} in general. When $\Gamma$ is a circle, fix an orientation and denote by $x_i$ (resp.\ $y_i$) the labels of   edges whose orientation agrees (resp.\ disagrees) with that of the circle; the absence of a proper plateau is equivalent to $ \prod x_i\wedge\prod y_i =1$. When $G$ is not large, finite index subgroups $\ov G$ of $G$ are   determined, up to isomorphism, by the degree of the topological covering $\ov \Gamma\to\Gamma$ (see \cite{Dud2, Wh}  for the case of $G=BS(m,n)$ with $m$ and $n$ coprime). It would be interesting to generalize the results of \cite{Bu,Dud2, Ge}  to these groups. 

\begin{proof}  
We may assume that $G$ is non-elementary.
Let $G$ be represented by  $\Gamma$. It is   large if $\beta (\Gamma)\ge2$. If $\Gamma$ is a tree, then $G$ is virtually $F_n\times \Z$  so is large (one may also use branched coverings to see this).  Assume therefore $\beta (\Gamma)=1$. If $\Gamma$ contains a proper plateau $P$ (in particular, if $\Gamma$ is minimal and has a terminal vertex),   a  branched  covering ramified over $P$ as in the proof of Proposition \ref{revet} yields an admissible map $ \ov\Gamma\to\Gamma$ with $\beta (\ov\Gamma)\ge2$, so $G$ is large.  
 On the other hand, since every connected space covering   a circle is a circle,  Corollary \ref{cov} implies that $G$ is not large if $\Gamma$ is a circle with no proper plateau.  
\end{proof}

\begin{prop} \label{papla}
Given a connected labelled graph $\Gamma$, there exists an admissible map $ \pi:\ov\Gamma\to\Gamma$ such that $\ov\Gamma$ is connected and contains no proper plateau. 

Every GBS group has a finite index subgroup represented by  a labelled graph with no proper plateau.
\end{prop}

\begin{proof} 

The second assertion follows from the first one and Lemma \ref{adm}, so we concentrate on the first.
Given $p$, we shall construct an admissible $\pi$ such that $\ov\Gamma$ is connected, $\ov\Gamma$ contains no proper $p$-plateau, and 
all multiplicities are powers of $p$.  
Noting that, for $p'\ne p $, there is no   proper $p'$-plateau in $\ov\Gamma$ if  there is none in $\Gamma$, the proposition follows by induction on the number of primes $p$ such that $\Gamma$ contains a proper $p$-plateau.

The idea is the following. Consider the union of all proper $p$-plateaux in $\Gamma$, and construct a branched covering ramified over it as in the proof of Proposition \ref{revet}. Then iterate. Since the graph grows, it is not obvious that this process terminates, so we prefer to work purely within $\Gamma$.

We construct graphs $\Gamma_i$ which are identical to $\Gamma$, but with different labels. 
Let   $\calp_1$ be the (disjoint) union of the proper $p$-plateaux of $\Gamma_1=\Gamma$. If $\calp_1\ne\es$, define a new labelled graph $\Gamma_2$ by dividing by $p$  the label of each oriented edge $e$ in $\bo \calp_1$ (see Definition \ref{bord}).  Define $\calp_2$ as the union of the proper $p$-plateaux of $\Gamma_2$, and iterate until obtaining $\Gamma_{r+1}$ containing no proper $p$-plateau (the process terminates because labels decrease). 

Given a vertex $v$  of $\Gamma$, define $P(v)$ as the number of $i\in \{1,\dots,r\}$ such that $v\in \calp_i$.
Define $P(e) $ similarly for an edge $e$ by counting the number of times that $e\inc\calp_i$. Note that $P(\tilde e)=P(e)$ if $\tilde e$ is the opposite edge. If $e$ has origin $v$, then  $P(v)- P(e)$  is the number of times that the label of $e$ gets divided by $p$ when passing from $\Gamma$ to $\Gamma_{r+1}$. In particular, $p^{P(v)-P(e)}$ divides the label $\lambda_e$ of $e$ in $\Gamma$.

We now describe $\pi$ by describing preimages. Its total multiplicity is $p^r$. The preimage    
of a vertex $v$ consists of $p^{r-P(v)}$ points, each of multiplicity $p^{P(v)}$. The preimage of an edge $e$ consists of $p^{r-P(e)}$ edges of multiplicity $p^{P(e)}$, each with label $\lambda_e/p^{P(v)-P(e)}$. If $e$ has origin $v$, exactly $p^{P(v)-P(e)}$ lifts of $e$ are attached to each preimage of $v$ (there is a choice in the way lifts of $e$ are attached to preimages of $v$).   

This defines a graph $\ov\Gamma$ and a map $\pi$.  We  check that $(*)$ is satisfied, and $\ov\Gamma$ has no proper $p$-plateau. 

Condition $(*)$ is equivalent to the equality $p^{P(v)-P(e)}=  \lambda_e\wedge p^{P(v)} $, for $e$ an edge with origin $v$. There are two cases. If $P(e)=0$, then $\lambda_e$ is divisible by $p^{P(v)}$, 
so $  \lambda_e\wedge p^{P(v)} =p^{P(v)}$. If $P(e)>0$, there exists $i\in \{1,\dots,r\}$ such that $p$ does not divide the label of $e$ in $\Gamma_i$,  
so $\lambda_e$ is divisible by $p^{P(v)-P(e)}$ but not by $p^{P(v)-P(e)+1}$. Thus $  \lambda_e\wedge p^{P(v)} =p^{P(v)-P(e)}$.

The label of an edge of   $\ov\Gamma$ depends only on its image in  $\Gamma$, and it is the label of that image in the graph $\Gamma_{r+1}$. Since $\Gamma_{r+1}$ has no proper $p$-plateau, it follows that the only possible  proper $p$-plateaux of $\ov\Gamma$ are its connected components. Restricting to a component of $\ov\Gamma$ yields the desired map $\pi$ (one may show that $\ov\Gamma$ is connected, but this is not needed).
\end{proof}

\begin{thm}\label{co} Given two GBS groups $G_1,G_2$ represented by labelled graphs $\Gamma_1,\Gamma_2$ which are strongly slide-free  and contain no proper plateau, one may decide whether $G_1$ and $G_2 $ have isomorphic finite index subgroups or not. 
\end{thm}

$\Gamma$ is \emph{strongly slide-free} \cite{Fodef} if, whenever two edges $e,f$ have the same origin, then $\lambda_e$ does not divide $\lambda_f$. There is at most one minimal strongly slide-free  labelled graph representing $G$ \cite{Fodef}.

The group $G=BS(m,n)$ satisfies the hypotheses of the theorem if and only if $m$ and $n$ are coprime 
(unless $G$ is   solvable).
 The theorem also applies to   GBS groups represented by a circle with no proper plateau. 
\begin{proof}
Consider  $G_1,G_2$  
as in the theorem. After replacing them by a subgroup of index 2 if needed, the modulus maps $\Delta_{G_i}$ only take positive values, so we may assume that all labels are positive (see Section \ref{prel}). We also assume that $\Gamma_1,\Gamma_2$ are minimal. Note that they have no terminal vertex (such a vertex would be a plateau). By a covering, we always mean a label-preserving topological covering between labelled graphs.

Suppose that $H$ is isomorphic to a finite index subgroup of both $G_1$ and $G_2$. By Corollary  \ref{cov}, it is represented by   labelled graphs $\ov\Gamma_i$ which   are coverings of $\Gamma_i$. These graphs are minimal and strongly slide-free, so by Theorem 1.2 of \cite{Fodef} (see also \cite{Gu}) they are the same. This means that  the labelled graphs $\Gamma_1$ and $\Gamma_2$ have a common finite covering if  $G_1$ and $G_2$ are commensurable. The converse is also true by Assertion 2 of Lemma \ref{adm}, so we are reduced to deciding whether   $\Gamma_1$ and $\Gamma_2$ have a common finite covering (as labelled graphs).

We claim that, given $\Gamma_1$ and $\Gamma_2$, we can find a number $M$ such that, if $\Gamma_1$ and $\Gamma_2$ have a common finite covering, then they have one with at most  $M$ edges. Assuming the claim, we can decide existence of a common finite covering by inspection.

We deduce the claim from the fact that the proof of Leighton's graph covering theorem \cite{Lei} is constructive. We follow the account given in \cite{Ne}, proof of Theorem 1.1. 

Assume that $\Gamma_1$ and $\Gamma_2$ have a common finite covering, and denote by $\tilde \Gamma$ their universal covering. We first subdivide edges $e$ of $\Gamma_1$ and $\Gamma_2$ having the same label at both ends,  
placing labels 1 near the created midpoint. This ensures that $\Aut(\tilde \Gamma)$ acts on $\tilde \Gamma$ without inversions. 

As  in \cite{Ne}, we let $C= \tilde \Gamma/\Aut(\tilde \Gamma)$ be the graph of colors. It is a quotient of $\Gamma_i$, so  belongs to an explicit finite list. For each graph in this list, we define numbers $n_i,m_k,r_k$ using $\Gamma_1$ as in \cite{Ne}, and we compute numbers $a_i$ and $b_k$. As shown in \cite{Ne}, $\Gamma_1$ and $\Gamma_2$ have a common covering whose number of edges may be explicity bounded in terms of $C,\Gamma_1,\Gamma_2,b_k$. Since there are only finitely many possibilities for $C$, the existence of $M$ follows. 
\end{proof}

 \begin{flushleft}

Gilbert Levitt\\
Laboratoire de Math\'ematiques Nicolas Oresme\\
Universit\'e de Caen et CNRS (UMR 6139)\\
BP 5186\\
F-14032 Caen Cedex\\
France\\
\emph{e-mail:} \texttt{levitt@unicaen.fr}\\

\end{flushleft}
 
 \end{document}